\documentclass{amsart}

\usepackage{amsmath}
\usepackage{amsxtra}
\usepackage{amscd}
\usepackage{appendix}
\usepackage{amsfonts}
\usepackage{amssymb}
\usepackage{mathtools}
\usepackage{eucal}
\usepackage[all]{xy}
\usepackage{graphicx}
\usepackage{tikz-cd}
\usepackage{mathrsfs}
\usepackage{hyperref}
\usepackage{color}
\usepackage{colonequals}
\usepackage{longtable}
\usepackage{float}
\usepackage{caption}
\usepackage{enumerate}
\usepackage{stmaryrd}
\usepackage{faktor}
\usepackage{bbm}
\usepackage{lineno}
\usepackage[new]{old-arrows}
\usepackage[normalem]{ulem}
\usepackage{enumitem}
\usepackage[colorinlistoftodos, textsize=tiny]{todonotes}
\def\listtodoname{List of Todos}
\def\listoftodos{\@starttoc{tdo}\listtodoname}
\usepackage[hyphenbreaks]{breakurl}
\usepackage[backref=true,
            style=alphabetic,
            doi=false]{biblatex}
\DeclareNameAlias{author}{given-family}

\usepackage{tikz}
\usepackage{tikz-cd}
\usetikzlibrary{matrix,arrows,decorations.pathmorphing}

\usepackage{kantlipsum} 

\calclayout

\usepackage{amsthm,aliascnt}
\usepackage{url}

\newcommand{\mynewtheorem}[2]{
  \newaliascnt{#1}{subsubsection}
  \newtheorem{#1}[#1]{#2}
  \aliascntresetthe{#1}
  \expandafter\def\csname #1autorefname\endcsname{#2}
}

\mynewtheorem{theorem}{Theorem}
\mynewtheorem{proposition}{Proposition}
\mynewtheorem{lemma}{Lemma}
\mynewtheorem{corollary}{Corollary}
\mynewtheorem{conjecture}{Conjecture}
\newtheorem*{conjecture*}{Conjecture}
\newtheorem*{theorem*}{Theorem}
\newtheorem*{ncs*}{Statement}
\mynewtheorem{fact}{Fact}
\mynewtheorem{mainthm}{Theorem}

\mynewtheorem{maindef}{Definition}

\mynewtheorem{mainrem}{Remark}

\theoremstyle{definition}
\mynewtheorem{exercise}{Exercise}
\mynewtheorem{definition}{Definition}
\mynewtheorem{example}{Example}
\mynewtheorem{remark}{Remark}

\newcommand{\one}{\mathbf{1}}

\newcommand{\C}{\mathbb{C}}

\newcommand{\F}{\mathbb{F}}

\newcommand{\Z}{\mathbb{Z}}

\newcommand{\calw}{\mathcal{W}}
\newcommand{\calf}{\mathcal{F}}

\DeclareMathOperator{\Stab}{Stab}

\DeclareMathOperator{\tr}{tr}
\DeclareMathOperator{\GL}{GL}

\renewcommand{\hom}{\mathrm{Hom}}

\DeclareMathOperator{\Res}{Res}
\DeclareMathOperator{\ind}{Ind}

\DeclareMathOperator{\Ind}{Ind}
\DeclareMathOperator{\Hom}{Hom}

\DeclareMathOperator{\End}{End}

\DeclareMathOperator{\stab}{Stab}

\DeclareMathOperator{\mat}{Mat}
\DeclareMathOperator{\gl}{GL}

\DeclareMathOperator{\SL}{SL}

\DeclareMathOperator{\ev}{ev}

\DeclareMathOperator{\bil}{Bil}
\DeclareMathOperator{\tril}{Tril}
\DeclareMathOperator{\Rep}{Rep}

\DeclareMathOperator{\id}{id}
\newcommand{\flbar}{\overline{\mathbb{F}}_{\ell}}

\DeclareMathOperator{\diag}{diag}
\DeclareMathOperator{\soc}{soc}

\newcommand{\wt}{\widetilde}
\newcommand{\wh}{\widehat}
\newcommand{\Pbar}{\overline{P}}
\newcommand{\Ubar}{\overline{U}}

\newcommand{\paren}[1]{\left(#1\right)}
\newcommand{\set}[1]{\left\{#1\right\}}

\newcommand{\verts}[1]{\left\lvert#1\right\rvert}

\newcommand\restr[2]{{
  \left.\kern-\nulldelimiterspace 
  #1 
  \right|_{#2} 
  }}

\makeatletter

\DeclareRobustCommand*{\lquotient}[3][]
{
  { \mathpalette{\mfaktor@impl@}{{#1}{#2}{#3}} }
}
\newcommand*{\mfaktor@impl@}[2]{\mfaktor@impl#1#2}
\newcommand*{\mfaktor@impl}[4]{
  \settoheight{\faktor@zaehlerhoehe}{\ensuremath{#1#2{#3}}}%
  \settoheight{\faktor@nennerhoehe}{\ensuremath{#1#2{#4}}}%
  \raisebox{-0.5\faktor@zaehlerhoehe}{\ensuremath{#1#2{#3}}}%
  \mkern-4mu\diagdown\mkern-5mu%
  \raisebox{0.5\faktor@nennerhoehe}{\ensuremath{#1#2{#4}}}%
}

\makeatother

\title[Mod $\ell$ gamma factors and a converse theorem]{Mod $\ell$ gamma factors and a converse theorem for finite general linear groups}

\author[\tiny{J. Bakeberg}]{Jacksyn Bakeberg}
\address[Jacksyn Bakeberg]{Department of Mathematics, Boston University}
\email{bakeberg@bu.edu}

\author[\tiny{M. Gerbelli-Gauthier}]{Mathilde Gerbelli-Gauthier}
\address[Mathilde Gerbelli-Gauthier]{Department of Mathematics and Statistics, University of Toronto}
\email{m.gerbelli@utoronto.ca}

\author[\tiny{H. Goodson}]{Heidi Goodson}
\address[Heidi Goodson]{Department of Mathematics, Brooklyn College, City University of New York; 2900 Bedford Avenue, Brooklyn, NY 11210 USA}
\email{heidi.goodson@brooklyn.cuny.edu}

\author[\tiny{A. Iyengar}]{Ashwin Iyengar}
\address[Ashwin Iyengar]{Department of Mathematics, Johns Hopkins University}
\email{ashwin.iyengar1@gmail.com}

\author[\tiny{G. Moss}]{Gilbert Moss}
\address[Gilbert Moss]{Department of Mathematics, University of Maine}
\email{gilbert.moss@maine.edu}

\author[\tiny{R. Zhang}]{Robin Zhang}
\address[Robin Zhang]{Department of Mathematics, Massachusetts Institute of Technology}
\email{robinz@mit.edu}

\date{August 14, 2025}

\begin{document}

\vspace*{-1.5cm}
\maketitle

\begin{abstract}
    The local converse theorem for Rankin--Selberg gamma factors of $\gl_2(\F_q)$
    proved by Piatetski-Shapiro over $\C$ no longer holds after reduction modulo $\ell \neq p$.
    To remedy this, we construct new $\gl_n \times \gl_m$ gamma factors valued in arbitrary $\mathbb{Z}[1/p, \zeta_p]$-algebras for Whittaker-type representations, show that they satisfy a functional equation,
    and then prove a $\gl_n \times \gl_{n-1}$ converse theorem for irreducible cuspidal representations.
    In the~$\gl_2 \times \gl_1$ case, we define an alternative ``new'' gamma factor,
    which takes values in~$k$ and satisfies a converse theorem that matches
    the converse theorem in characteristic $0$. 
\end{abstract}

\tableofcontents


\section{Introduction}

Let $\F_q$ be a finite field of order $q$ and characteristic $p$, and let $\ell$ be a prime different from $p$. In the $\ell$-modular representation theory of finite groups such as $\gl_n(\mathbb{F}_q)$, the importance of tools such as Brauer theory and Deligne--Lusztig varieties is well-established \cite{fong_srinivasan, dipper_james}. In this paper, we investigate a different tool, inspired by a construction in the local Langlands program for $p$-adic groups: gamma factors. While gamma factors first arose in the context of complex representations, they have been fruitful in studying $\ell$-modular representations of~$\gl_n(\mathcal{F})$, where $\mathcal{F}$ is a $p$-adic field with residue field~$\F_q$ (\cite{helm_moss_converse, moss}).  

Fix a nontrivial character $\psi:\F_q\to \C^{\times}$. Complex Rankin--Selberg gamma factors $\gamma(\pi \times \pi',\psi)\in \C^{\times}$ have been studied for pairs $\pi$, $\pi'$ where $\pi$ is a complex representation of $\gl_n(\F_q)$, $\pi'$ is a complex representation of $\gl_m(\F_q)$, and both $\pi$ and $\pi'$ are assumed to be irreducible and $\psi$-generic \cite{ps,roditty,nien,ye,ye-zelingher,zelingher,soudry-zelingher}. In this context, there are converse theorems inspired by the $p$-adic setting, which describe sets of $\pi'$ such that $\gamma(\pi\times \pi',\psi)$ uniquely determine $\pi$. 

Let $k$ be an algebraically closed field of characteristic $\ell$ or $0$. Our main result is a converse theorem for irreducible cuspidal $k$-representations. 
It uses a gamma factor valued in certain finite dimensional local $k$-algebras instead of $k$ itself.

\begin{maindef}
    \label{def:new-gamma-factor}
    For an irreducible generic $k$-representation $\pi'$ of $\gl_m(\F_q)$ with projective cover~$P(\pi')$,
    let $R(\pi') := \End_{k[\gl_m(\F_q)]}(P(\pi'))$. Define the gamma factor
    \[
        \widetilde{\gamma}(\pi\times\pi',\psi) := \gamma\Big(\big(\pi\otimes_k R(\pi')\big)\times P(\pi'),\psi\Big)\in R(\pi')^{\times}\ .
    \]
\end{maindef}

These $R(\pi')^\times$-valued gamma factors are defined by a functional equation, see \autoref{mainthm_functionaleqn}. We show that they satisfy:

\begin{mainthm}\label{mainthm_converse}
    Let $\pi_1$ and $\pi_2$ be irreducible cuspidal $k$-representations of $\gl_n(\F_q)$.
    If
    \[
        \widetilde{\gamma}(\pi_1\times \pi',\psi) = \widetilde{\gamma}(\pi_2\times \pi',\psi)
    \]
    for all irreducible generic $k$-representations $\pi'$ of $\gl_{n-1}(\F_q)$, then $\pi_1\cong \pi_2$. 
\end{mainthm}

We can also use the functional equation of \autoref{mainthm_functionaleqn} to define a simpler, ``na\"ive" gamma factor $\gamma(\pi \times \pi',\psi)$ without reference to exotic $k$-algebras (this is implicit in the use of the notation ``$\gamma$" in \autoref{def:new-gamma-factor}). Irreducible cuspidal representations in characteristic $\ell$ arise via reduction mod $\ell$ of integral lattices in representations defined over $\overline{\mathbb{Q}}$ \cite[III.2.2]{vig_book}, and the gamma factor $\gamma(\pi\times\pi',\psi)$ is simply the reduction mod $\ell$ of the associated gamma factor in $\overline{\mathbb{Z}}^{\times}$ (see \autoref{compatibility_with_specialization}, below). When $\ell\nmid \#\gl_n(\F_q)$, the representation theory of $\gl_n(\F_q)$ over $k$ is essentially no different than the complex setting: we have $P(\pi')=\pi'$, so $\widetilde{\gamma}(\pi\times\pi',\psi) = \gamma(\pi\times\pi',\psi)$, and \autoref{mainthm_converse} shows $\gamma(\pi\times\pi',\psi)$ satisfies a converse theorem
. However, when $\ell|\#\gl_n(\F_q)$, the na\"ive gamma factor $\gamma(\pi\times\pi',\psi)$ fails to satisfy the converse theorem in several examples when $n=2$. 

Using SAGE computations (\autoref{sec:sage}), we found that the converse theorem for $\GL_2(\F_q)$ fails when $ (\ell, q) = (2,5)$, $(2,17)$, $(3,7)$, $(3,19)$, $(5,11)$, $(11,23)$, $(23,47)$, $(29,59)$, though we verified it for all other pairs $(\ell,q)$ with $\ell \leq 11$ and $q = p \leq 23$. In all counterexamples we found, $q$ has the form $2\ell^k+1$, and we conjecture these are the only cases in which it can fail. The point of failure in the classical proof of the converse theorem is the failure of so-called ``$L^2$-completeness'' of the Whittaker space; this was observed in the $p$-adic setting in \cite{moss}.

In light of these counterexamples, \autoref{mainthm_converse} provides an answer to the natural question of how to uniformly construct a ``new'' gamma factor for all $\ell \neq p$ that satisfies a converse theorem in general and returns the classical gamma factor when $\ell \nmid \#\gl_n(\F_q)$. In \autoref{sec:computing new gamma} we compute the new gamma factors in the case $n=2$, $(\ell,q)=(2,5)$ and use the computation to illustrate \autoref{mainthm_converse} in this case.

The key to proving \autoref{mainthm_converse} is the observation that one can recover $L^2$-completeness of Whittaker models if one is allowed to pair with $R(\pi')$-valued Whittaker functions of projective envelopes $P(\pi')$ as $\pi'$ varies over generic irreducibles (see \autoref{completeness}). In order to make use of the classical arguments usually used to prove converse theorems, it remains to extend the functional equation defining $\gamma(\pi\times\pi',\psi)$ to rings of the form $R(\pi')$. In fact, we generalize it to arbitrary $\mathbb{Z}[1/p, \zeta_p]$-algebras $A$ where $\zeta_p$ is a primitive $p$-th root of unity and $\pi, \pi'$ are Whittaker-type representations of $\gl_n(\F_q)$ and $\gl_m(\F_q)$ on $A$-modules. Unlike the $p$-adic setting, there are some $\pi'$ that are ``exceptional'' for $\pi$ and for which the functional equation may fail; we use Bernstein--Zelevinsky derivatives to precisely describe these exceptional $\pi'$. Our construction of the gamma factor is compatible with extensions of scalars along ring homomorphisms $A\to A'$, returns the gamma factor of \autoref{def:new-gamma-factor} when $A = R(\pi')$, and returns the classical gamma factors when $A=\C$.

Bernstein and Zelevinsky developed a theory of ``derivatives'' for complex representations of $\gl_n(\mathcal{F})$ \cite{bz77} with respect to a fixed additive character $\psi$ on $\mathcal{F}$. Fixing $\psi:\F_q\to \mathbb{Z}[\frac{1}{p},\zeta_p]^{\times}$, Vign\'{e}ras observed that derivatives work equally well for $\gl_n(\F_q)$-representations on $A$-modules. If $\pi$ is an $A[\gl_n(\F_q)]$-module its ``$i$-th derivative''  $\pi^{(i)}$ is a representation of $\gl_{n-i}(\F_q)$, and the restriction $\pi|_{P_n}$ to the mirabolic subgroup $P_n$ (matrices with bottom row $(0,\dots,0,1)$) is glued from $\pi^{(1)},\dots, \pi^{(n)}$ in a simple way. The top derivative $\pi^{(n)}$ is equivalent to the $(N_n,\psi)$-coinvariants, where $N_n$ is the unipotent upper triangular subgroup, to which $\psi$ is extended in a natural way. Thus $(\pi^{(n)})^{\vee}$ is the space of Whittaker models of $\pi$, by Frobenius reciprocity. The starting point of our construction is to restrict our attention to $\pi$ of \emph{Whittaker type}, meaning $\pi^{(n)}$ (and hence $(\pi^{(n)})^{\vee}$) is free of rank one over $A$, which generalizes the ``irreducible generic'' hypothesis ubiquitous in the $A = \mathbb{C}$ case. In particular, this allows one to speak of the Whittaker model $\mathcal{W}(\pi,\psi_A)\subset \Ind_N^G\psi_A$, where $\psi_A = \psi\otimes_{\mathbb{Z}[\frac{1}{p},\zeta_p]}A$. 

Define an $A[\gl_m(\F_q)]$-module $\pi'$ to be \emph{exceptional} for~$\pi$ if there exists $k=1,\dots, m$ such that
\[
    \Hom_{A[\GL_k(\F_q)]}\left(\mathcal{W}(\pi,\psi_A)^{(n-k)},\left(\mathcal{W}(\pi',\psi^{-1}_A)^{(m-k)}\right)^{\vee}\right)\neq 0.
\]
We obtain a functional equation that is more general when specializing to $k = \C$
than any appearing in the literature because $\pi$ need not be cuspidal, nor even irreducible.
This gives a construction of gamma factors of which \autoref{def:new-gamma-factor} is a special case.

\begin{mainthm}
\label{mainthm_functionaleqn}
Suppose $\pi$ and $\pi'
$ are Whittaker type $A[\gl_n(\F_q)]$- and $A[\gl_m(\F_q)]$-modules, respectively, and $\pi'$ is not exceptional for $\pi$. There exists a unique element $\gamma(\pi\times\pi',\psi)$ of $A^{\times}$ such that
\begin{multline}\label{functional_equation}
\left(\sum_{x\in N_m\backslash G_m}W\left(\begin{smallmatrix}x & 0\\ 0 & I_{n-m}\end{smallmatrix}\right)W'(x)\right)\gamma(\pi\times \pi',\psi) \\
= \sum_{x\in N_m\backslash G_m}\sum_{y\in M_{m, n-m-1}}W\left(\begin{smallmatrix}0 & 1 &0\\
0&0 & I_{n-m-1}\\ x &0 & y\end{smallmatrix} \right)W'(x)
\end{multline}
for all $W\in \mathcal{W}(V,\psi_A)$, $W'\in \mathcal{W}(V',\psi_A^{-1})$.
\end{mainthm}

When $A=k$ is a field, an irreducible representation $\pi$ of $\gl_n(\F_q)$ is cuspidal if and only if $\pi^{(n)}$ is its only non-zero derivative, in which case there are no exceptional representations~$\pi'$. The functional equation in \autoref{mainthm_functionaleqn} is known to fail when the non-exceptional hypothesis for $\pi'$ is removed by \cite[Theorem 3.1]{roditty}.

In \autoref{sec: ell-regular gamma factors}, we propose an alternative ``new gamma factor'' for $n=2$ and $m=1$, defined over the base field $k$, which also specializes to the classical one for $\ell\nmid \#\gl_n(\F_q)$ but without involving Artinian local $k$-algebras. Remarkably, it also satisfies a functional equation and converse theorem for cuspidals. This method shares some similarities with \cite{vigneras_new_epsilon}, including the fact that it does not appear to generalize beyond $n=2$.

\subsection{Future directions}

\subsubsection{Macdonald correspondence in families}

In \cite{macdonald}, Macdonald established an analogue of the local Langlands correspondence for~$\gl_n(\F_q)$. If $\mathcal{F}$ is a nonarchimedean local field with residue field~$\F_q$, it can be formulated as a bijection between complex irreducible representations of~$\gl_n(\F_q)$ and tame inertial classes of complex representations of the Weil group of~$\mathcal{F}$. This bijection preserves gamma factors, which (following \cite{springer_zeta}) Macdonald defined analogously to the Godement--Jacquet factors for representations of $\gl_n(\mathcal{F})$. Later, Vign\'{e}ras found a similar but more subtle bijection in the mod $\ell$ setting (\cite[Section 3]{vigneras_luminy}), but she did not consider gamma factors. More recently, a local Langlands correspondence ``in families'' has been established for $\gl_n(\F_q)$. If $\mathcal{O}$ is a complete discrete valuation ring with residue field $\overline{\mathbb{F}}_{\ell}$, it takes the form of an isomorphism of commutative rings $B_{q,n}\xrightarrow{\sim} \End_{\mathcal{O}[\gl_n(\F_q)]}(\ind_{N_n}^{\gl_n(\F_q)}\psi_{\mathcal{O}})$, where $B_{q,n}$ is the ring of functions on a natural moduli space of tame $\ell$-adically continuous $\mathcal{O}$-valued inertial classes, and $\psi_{\mathcal{O}}:N_n\to \mathcal{O}^{\times}$ is a nondegenerate character on the unipotent upper triangular subgroup $N_n$. The first approach to proving the existence of such an isomorphism was to deduce it as a consequence of the local Langlands correspondence in families for~$\gl_n(\mathcal{F})$ \cite{helm_moss_converse,helm_curtis} (see also \cite{shotton}) which, in turn, requires gamma factors, converse theorems, and the classical local Langlands correspondence for $\gl_n(\mathcal{F})$, as an input. More recently, Li and Shotton found a remarkable second proof of ``finite fields local Langlands in families,'' which works for any reductive group $G(\F_q)$ whose dual group has simply connected derived subgroup. Their proof uses purely finite fields methods (\cite{li, li_shotton}), but they do not consider gamma factors. The present paper is a first step toward understanding how Rankin--Selberg gamma factors for $\gl_n(\F_q)$ fit into the $\ell$-modular correspondence and the families correspondence.

In future work, the authors plan to apply the converse theorems proved here to address the question of whether the Macdonald bijection and its mod $\ell$ analog are the unique sequence of bijections (one for each~$n$) matching the Rankin--Selberg gamma factors $\gamma(\pi\times\pi',\psi)$ defined here with Deligne's $\epsilon_0$ factors of the tensor product of the corresponding inertial classes of $W_{\calf}$-representations. To show the local Langlands correspondence for $\gl_n(\F_q)$ in families preserves our new gamma factors (and is uniquely characterized by this property), one would need to establish a compatibility between Curtis homomorphisms, which were used in \cite{helm_curtis} and \cite{li_shotton} to construct the local Langlands in families, and the Rankin--Selberg gamma factors. It seems a multiplicativity property for our gamma factors would be needed here.

\subsubsection{Converse theorem for generic irreducibles}
Our proof of \autoref{mainthm_converse} only works for irreducible cuspidals because such representations $\pi$ have no exceptional representations $\pi'$. When~$\pi'$ is exceptional for $\pi$, the gamma factor $\gamma(\pi\times\pi',\psi)$ can no longer be defined using a functional equation, but in the complex setting it is traditionally defined using Bessel vectors (\cite{roditty,nien}). Recently, Soudry and Zelingher \cite{soudry-zelingher} proved a multiplicativity property for $\gamma(\pi\times\pi',\psi)$ thus defined, and used it to deduce a converse theorem applying to all irreducible generic representations $\pi$. In future work we plan to investigate the question of whether this remains true in the mod $\ell$ setting, either by establishing a generalization of the Bessel vector construction and the multiplicativity property to characteristic $\ell>0$, or by using the functional equation even while excluding exceptional $\pi'$.

\subsubsection{Jacquet's conjecture in the mod $\ell$ setting}
In the complex setting, it has been proved that $\pi$ is characterized by gamma factors $\gamma(\pi\times\pi',\psi)$ where $\pi'$ ranges over irreducible generic representations of $\gl_m(\F_q)$ for $m=1,\dots,\lfloor \frac{n}{2}\rfloor$ (\cite{nien} for irreducible cuspidal $\pi$, and \cite{soudry-zelingher} in general). It is natural to ask whether our \autoref{mainthm_converse} remains true with $m\leq \lfloor\frac{n}{2}\rfloor$; the answer is probably yes (c.f. \cite{moss} in the $p$-adic setting), but we do not to address this here.


\subsection*{Acknowledgments}
This project began as part of the Rethinking Number Theory 2 workshop. The authors are deeply grateful to the organizers for the mathematical experience and the welcoming community. The workshop, as well as the continued work on this project after the workshop, was generously supported by the Number Theory Foundation and the American Institute of Mathematics.

The authors are grateful to Rob Kurinczuk for several helpful conversations and to an anonymous referee whose suggestions improved the exposition and quality of the paper.

J.B. was supported by National Science Foundation Grants DGE-1840990 and DGE-2234657, 
H.G. was supported by National Science Foundation Grant DMS-2201085 and a PSC-CUNY Award, jointly funded by The Professional Staff Congress and The City University of New York,
G.M. was supported by National Science Foundation Grants DMS-2001272 and DMS-2234339, 
R.Z. was supported by National Science Foundation Grants DGE-1644869 and DMS-2303280.


\section{Preliminaries}

In this section we collect some basic facts about the representation theory of $\GL_n(\F_q)$.

\subsection{Subgroups of \texorpdfstring{$\gl_n(\F_q)$}{GL(n, Fq)}}

Let $p$ be a prime, $q$ be a power of $p$, and $\F_q$ be the field with~$q$ elements. For $n$ a positive integer, let \[ G_n \colonequals  \gl_n(\F_q). \] Denote by $\mat_{m_1,m_2}(\F_q)$ the vector space of $m_1\times m_2$ matrices over $\F_q$. 
 The mirabolic subgroup of $G_n$ is
    \[ P_n \colonequals \set{\begin{pmatrix} g & y \\ 0 & 1 \end{pmatrix} \in G_n : g \in G_{n-1}, y \in \mat_{n-1,1}(\F_q)}, \]
with unipotent radical  \[U_n \colonequals  \set{\begin{pmatrix} I_{n-1} & y \\ 0 & 1 \end{pmatrix} : y \in \mat_{n-1,1}(\F_q)} \leq P_n, \] so that
   $ P_n = U_n \rtimes G_{n-1}$. We also denote \[ N_n \colonequals  \left(\begin{array}{cccc}
     1& * &\hdots & *\\
     0& 1 & \ddots & \vdots \\
     \vdots& \ddots&\ddots &* \\
     0 & \hdots & 0& 1
\end{array}\right) \] the subgroup of unipotent upper-triangular matrices in $G_n$.  We consider a sequence of subgroups interpolating between $U_n$ and $N_n$: for $-1 \leq m \leq n-1$, define $$U_{n,k}:=\left\{\begin{pmatrix} I_{n-k} & z\\ 0 & y\end{pmatrix}\ :\ z\in \text{Mat}_{n-k, k}(\mathbb{F}_q),\ y\in N_{k} \right\}.$$ 
Note that \[ I_n = U_{n,0}, \quad U_n = U_{n,1}, \quad N_n = {U_{n,n-1}} = U_{n,n}. \] 
\subsection{Representations}\label{section_representations}
Let $G$ be a finite group. In this article, our coefficient rings $R$ will always be assumed to be algebras over $\mathbb{Z}[\frac{1}{p},\zeta_p]$. Let $\Rep_R(G)$ denote the category of $R$-linear representations, or, equivalently, of $R[G]$-modules. An object of $\Rep_R(G)$ is denoted variously as an $R[G]$-module $V$, as a pair $(\pi,V)$ where $V$ is an $R$-module and $\pi: G \to \mathrm{Aut}_R(V)$ is a homomorphism, or simply as $\pi$ when the $R$-module $V$ is clear from context. We warn the reader that we will often use the letter $V$ to denote an $R[G]$-module, even if $V$ is not necessarily free as an $R$-module. If $H \leq G$ is a subgroup, the induction functor $\Ind_H^G: \Rep_R(H) \to \Rep_R(G)$ sends $(\pi,V)$ to the representation
    \[ \Ind_H^G(\pi) = \set{f: G \to V : f(hg) = \pi(h)f(g), \; h \in H}, \]
with its natural left $G$-action by right multiplication on $G$. Frobenius reciprocity is the statement that induction is a left-adjoint to restriction: given~$\rho \in \Rep_R(H)$ and $\pi \in \Rep_R(G)$,
\[ \Hom_G(\Ind_H^G\rho, \pi) \cong \Hom_H(\rho, \pi|_H). \]
The group ring $R[G]$ is equipped with a natural left $H$-action, which makes $\Ind_H^G(\pi)$ naturally isomorphic to $\Hom_{R[H]}(R[G],\pi)$ as left $R[G]$-modules, which some authors call ``coinduction''. However the distinction is unimportant because of the isomorphism given by \begin{align*}
R[G]\otimes_{R[H]}\pi &\xrightarrow{\sim} \Ind_H^G(\pi)\\
1\otimes v &\mapsto f_v ,
\end{align*}
where $v$ is an element in the space of $\pi$ and $f_v$ is the function supported on $H$ such that $f_v(h) = \pi(h)v$, $h\in H$. In particular, induction is also a right adjoint to restriction.

If $N \leq G$ is a subgroup such that $\verts{N}$ is invertible in $R$, and $\psi: N\to R^{\times}$ is a character, we define a projector to the submodule $\pi^{N,\psi}$ of elements on which~$N$ acts via $\psi$:
\begin{align}\label{coinvariants_projector}
\pi &\to \pi^{N,\psi}\\
v&\mapsto |N|^{-1}\sum_{n\in N} \psi(n)^{-1}\pi(n)v. \nonumber
\end{align}
The kernel of this projector equals the submodule $V(N,\psi)$ generated by $\{\pi(n)v-\psi(n)v:n\in N, v\in V\}$, so $\pi^{N,\psi}$ is canonically isomorphic to the $(N,\psi)$-coinvariants $\pi_{N,\psi}\colonequals V/V(N,\psi)$. 

Since induction is isomorphic to coinduction for representations of finite groups $H \leq G$, $(\Ind_{H}^{G} \sigma)^\vee\cong \Ind_{H}^{G} \sigma$ by Frobenius reciprocity. We use this to show that the dual of the $(N, \psi)$-coinvariants is the $(N, \psi^{-1})$-coinvariants of the dual
(see \cite[\S I.5.11]{vig_book} for locally profinite groups in general).

\begin{lemma}\label{top_deriv_dual}
Let $N$ be a subgroup whose cardinality is invertible in $R$ and let~$\psi:N\to R^{\times}$ be a character. Then for all 
$\pi \in \Rep_R(G)$, we have $(\pi_{N,\psi})^{\vee} \cong (\pi^{\vee})_{N,\psi^{-1}}$.
\end{lemma}
\begin{proof}
We make use of the fact that for a finite group $G$, $\Hom_{R[G]}(V,W^{\vee})\cong \Hom_{R[G]}(W,V^{\vee})$ (\cite[\S I.4.13]{vig_book}), where $(-)^{\vee}$ denotes the $R$-linear dual equipped with its natural $G$-action. Applying this, 
we have the following identifications
\begin{align*}
(\pi_{N,\psi})^{\vee} &\overset{\text{def}}{=} \Hom_{R}(\pi_{N,\psi},R)\\
&= \Hom_{R[N]}(\pi,\psi)\\
&=\Hom_{R[G]}(\pi,\Ind_N^G\psi)\\
&=\Hom_{R[G]}(\Ind_N^G\psi^{-1},\pi^{\vee}) \\
&=\Hom_{R[N]}(\psi^{-1},\pi^{\vee})\\
&=(\pi^{\vee})^{N,\psi^{-1}}\\
&=(\pi^{\vee})_{N,\psi^{-1}}. \qedhere
\end{align*}
\end{proof}

 Given a nontrivial partition $n_1 + \cdots + n_r$ of $n$, there is an associated standard parabolic subgroup $P_{n_1,...,n_r}$ with Levi subgroup $G_{n_1} \times \cdots \times G_{n_r}$. If $\sigma_i$ is a representation of $G_{n_i}$, then the parabolic induction    \[ \sigma_1 \times \cdots \times \sigma_r := \ind_{P_{n_1,...,n_r}}^{G_n} \sigma_1 \boxtimes \cdots \boxtimes \sigma_r \]
is obtained by first inflating $\sigma_1 \boxtimes \cdots \boxtimes \sigma_r$ to a representation of $P_{n_1,...,n_r}$ by letting its unipotent radical act trivially, and inducing the resulting representation to $G_n$.

The corresponding ``parabolic restriction" functors are known as Jacquet functors. Given a partition as above, the functor $J^{G_n}_{P_{n_1,...,n_r}}: \Rep_R(G_n) \to \Rep_R(G_{n_1} \times \cdots \times G_{n_r})$ takes a representation $(\pi,V) \in \Rep_R(G)$ to its coinvariants under the unipotent radical of $P_{n_1,...,n_r}$. 
The functor $J^{G_n}_{P_{n_1, ..., n_r}}$ is both left- and right-adjoint to parabolic induction. 

We say that $(\rho,V) \in \Rep_R(G_n)$ is \textit{cuspidal} if its image under the Jacquet functor~$J^{G_n}_{P_{n_1, \dots, n_r}}$ is zero for every non-trivial partition. This is equivalent to asking that there are no nonzero morphisms from $\rho$ to a parabolic induction. 

\subsection{Multilinear forms}

Gamma factors are defined as the constants of proportionality between certain multilinear forms, once the spaces of such forms are shown to be one-dimensional. We define those spaces now.

\begin{definition}
If $G$ is a group, $(\rho,V),(\rho',V'),(\rho'',V'') \in \Rep_R(G)$ and $\chi: G \to R^\times$ is a linear character, let
\begin{align*}
    \bil_G(V,V',\chi) &:= \Hom_{R[G]}(V \otimes_R V', \chi) \\
    &= \{\text{bilinear functions }B:V\times V'\to R \ |\ B(gv, gv')=\chi(g)B(v,v')\}.
\end{align*}
and let
\begin{align*}
    \tril_G(V,V',V'') &:= \Hom_{R[G]}(V \otimes_R V' \otimes_R V'', 1) \\
    &= \{\text{$G$-invariant trilinear functions }B:V \times V' \times V'' \to R \}.
\end{align*}
\end{definition}

In the above definitions $G$ acts diagonally on the tensor products.

\subsection{Derivative functors}
\label{ss derivative functors}
Let $R$ be a Noetherian commutative $\mathbb{Z}[\frac{1}{p}, \zeta_p]$-algebra such that $0 \neq 1$ in $R$. Fix once and for all a nontrivial group homomorphism $\psi:\mathbb{F}_q \to \mathbb{Z}[\frac{1}{p}, \zeta_p]^{\times}$ and denote by~$\psi_R$ its extension to $R^{\times}$ along the structure morphism $\mathbb{Z}[\frac{1}{p}, \zeta_p]\to R$. Promote $\psi_R$ to a character of $U_n$ (also denoted $\psi_R$ by abuse of notation) by letting
\[
    \psi_R \begin{pmatrix}
        I_{n-1} & y \\ 
        0&1
    \end{pmatrix} = \psi_R(y_{n-1}), \quad y = (y_1,...,y_{n-1})^t.
\]

To analyze representations of the mirabolic subgroup $P_n$, we recall \textit{derivative functors}, following Bernstein--Zelevinsky \cite{bz77} for $p$-adic general linear groups.

Specifically, define the functors
\begin{center}
    \begin{tikzcd}
        \Rep_R(P_{n-1}) \ar[r,  yshift=0.7ex, "\Phi^+"] & \ar[l,  yshift=-0.7ex, "\Phi^{-}"] \Rep_R(P_n) \ar[r,  yshift=-0.7ex, "\Psi^{-}"'] & \ar[l,  yshift= 0.7ex, "\Psi^{+}"'] \Rep_R(G_{n-1})
    \end{tikzcd}
\end{center}
where

\begin{itemize}
\item $\Psi^-(V) = V/V(U_n,\one)$ where $V(U_n,\one) = \langle\{uv - v: u\in U_n, v\in V\}\rangle$. It carries an action of $G_{n-1}$.

\item $\Psi^+(V) = V$ and we inflate the $G_{n-1}$ action to a $P_n$ action by letting $U_n$ act trivially.

\item $\Phi^-(V) = V/V(U_n,\psi_R)$ where $V(U_n,\psi_R) = \langle\{uv - \psi_R(u)v: u\in U_n, v\in V\}\rangle$. It carries an action of $P_{n-1}$ because $P_{n-1}$ is the stabilizer in $G_{n-1}$ of the character $\psi_R$ of $U_n$ under the conjugation action defined by $\psi_R\mapsto \psi_R(g(-)g^{-1})$.

\item $\Phi^+(V) = \ind_{P_{n-1}U_n}^{P_n}(V\otimes\psi_R)$ where $V\otimes\psi_R$ denotes the representation of $V$ extended to $P_{n-1}U_n$ by letting $U_n$ act via $\psi_R$. Since $P_{n-1}$ is the normalizer of $\psi_R$ this is well-defined.
\end{itemize}

\subsubsection{Properties of derivative functors} \label{ss properties of BZ functors}
Bernstein--Zelevinsky established some basic properties of these functors over $p$-adic general linear groups, and Vign\'{e}ras has observed that the proofs work equally well in the case of finite general linear groups \cite[\S III.1.3]{vig_book}. The properties we will need are the following:
\begin{enumerate}[label = (\Roman*)]
\item They are all exact.
\item $\Psi^-$ is left adjoint to $\Psi^+$
\item $\Phi^+$ is left adjoint to $\Phi^-$ and $\Phi^-$ is left adjoint to $\Phi^+$.
\item $\Phi^-\Phi^+\cong \id$ and $\Psi^-\Psi^+\cong \id$ 
\item $\Phi^-\Psi^+ = 0$ and $\Psi^-\Phi^+ = 0$
\item There is a canonical exact sequence
$$0\to \Phi^+\Phi^- \to \id \to \Psi^+\Psi^-\to 0.$$
\end{enumerate}
We note the following additional property:

\begin{enumerate}[resume]
\item[(VII)] All the functors commute with arbitrary base change. In other words, if $R\to R'$ is a map of rings, then $\Phi^+(V\otimes_{R}R') = \Phi^+(V)\otimes_RR'$, and the same for all the other functors.
\end{enumerate}

\begin{definition}
Given $V\in \Rep_R(P_n)$, define the ``$k$-th derivative''
    \[ V^{(k)} = \Psi^-(\Phi^-)^{k-1}(V), \]
which is in $\Rep_R(G_{n-k})$.  For $V\in \Rep_R(G_n)$, $V^{(k)}$ is defined to be the $k$-th derivative of its restriction to $P_n$. Finally, we define $V^{(0)}=V$ for $V\in \Rep_R(G_n)$.
\end{definition}

By successive application of property (VI) above, any $V\in \Rep_R(P_n)$ has a natural filtration by $P_n$-submodules:
\begin{equation}\label{bzfiltration}
0\subset V_n\subset V_{n-1} \subset \cdots \subset V_2\subset V_1 = V,
\end{equation}
where $V_k = (\Phi^+)^{k-1}(\Phi^-)^{k-1}(V)$. The successive quotients can be recovered from the derivatives of $V$ as follows:
\begin{equation}\label{eq quotients from derivatives}  V_k/V_{k+1} = (\Phi^+)^{k-1}\Psi^+(V^{(k)}).\end{equation}
This indicates the following remarkable fact: every representation of $P_n$ is ``glued together from'' representations of various $G_m$'s for $m<n$.

The next two lemmas give explicit descriptions of the derivative functors in terms of coinvariants and parabolic restriction. 

Let $k \leq n-1$. Extend $\psi_R$ to a character of $U_{n,k}$ via the map \begin{align*} U_{n,k}  \to U_{n,k}/[U_{n,k},U_{n,k}]  \cong  \F_q^{k} &\to \F_q \\
(y_1,...,y_k) &\mapsto y_1+ .. + y_k, \end{align*} so that in input of $\psi_R$ is the sum of all the upper-diagonal entries, $n-k-1$ of which are zero.

\begin{lemma} \label{top derivative explicit}
Suppose $(\rho,V) \in \Rep_R(P_n)$. Then 
$(\Phi^-)^k V \cong V_{U_{n,k},\psi_R}$, the space of $(U_{n,k},\psi_R)$-coinvariants. In particular, the $n$-th derivative $V^{(n)} \cong V_{N_n,\psi_R}$.
\end{lemma}
\begin{proof} 
Recall that for a subgroup $H$ of $G_n$,  \[ V_{H,\psi_R} = V/V(H,\psi_R):= V/\langle\{\psi_R(h)v-hv,  H \in N_n, v \in V\}\rangle. \]
We argue by induction on $k$. If $k=1$, by definition $\Phi^- V = V_{U_n,\psi_R} = V_{U_{n,1},\psi_R}$. 

Next, by induction, $(\Phi^{-})^{k}V = (V_{U_{n,k-1},\psi_R})_{U_{n-k,1}}$ where $U_{n-k,1}$ is embedded in the upper-left diagonal block, so in order to establish that $(\Phi^-)^kV\cong V_{U_{n,k},\psi_R}$, it suffices to show that \[  V(U_{n,k},\psi_R) = V(U_{n,k-1},\psi_R) \oplus V(U_{n-k,1},\psi_R). \] Since $U_{n,k-1}$ and $U_{n-k,1}$ are subgroups of $U_{n,k}$, the $\supset$ inclusion is immediate. For the reverse inclusion, observe that $U_{n,k} = U_{n,k-1} \rtimes U_{n-k,1}$ and that $U_{n-k,1}$ centralizes $\psi_R: U_{n,k-1} \to R$. So for $u \in U$ with $u=xy$ for $x \in U_{n,k-1}$, $y \in U_{n-k,1}$, and $v \in V$ we have \begin{align*}
     \psi_R(xy)v-(xy)v &= \psi_R(xy)v-\psi_R(x)yv + \psi_R(x)yv - (xy) v \\ & = \psi_R(yx)v-y\psi_R(x)v + \psi_R(x)yv - (xy) v \in V(U_{n,k-1}\psi_R) \oplus V(U_{n-k,1}\psi_R), 
\end{align*}
which provides the reverse inclusion. 

The last statement follows from the definition of derivatives, since $N_n = U_{n,n-1}$ and $\Psi^-:\Rep(P_1) \to \Rep(G_0)$ is the identity.
\end{proof}

\begin{lemma}[III.1.8 \cite{vig_book}]\label{deriv_parabolic_restriction}
The $k$-th derivative functor $\pi \mapsto \pi^{(k)}$ is the composite of parabolic restriction $J^{G_n}_{P_{n-k,k}}$ from $R[G_n]$-modules to $R[G_{n-k} \times G_k]$-modules with the top derivative from $R[G_k]$-modules to $R$-modules.
\end{lemma}

To emphasize the dependence on $\psi_R$ let us write $\pi^{(k,\psi_R)} = \pi^{(k)}$. In this notation, we have:
\begin{corollary}\label{derivative_dual}
Let $\pi$ be an $R[G_n]$-module, and let $1\leq k\leq n$. We have $(\pi^{(k,\psi_R)})^{\vee} \cong (\pi^{\vee})^{(k,\psi_R^{-1})}$.
\end{corollary}
\begin{proof}
When $k=n$ this follows from \autoref{top_deriv_dual} and \autoref{top derivative explicit}. When $k<n$, \autoref{top_deriv_dual} combined with \autoref{deriv_parabolic_restriction} shows that it suffices to prove the parabolic restriction functor commutes with duals. However since parabolic restriction is both left and right adjoint to parabolic induction, and parabolic induction commutes with duals (\cite[\S I.5.11]{vig_book}), it follows that parabolic restriction commutes with duals.
\end{proof}

The following is a characterization of restrictions of irreducible cuspidals in terms of Bernstein--Zelevinsky derivatives.

\begin{theorem}[III.1.5 \cite{vig_book}] \label{CuspidalCriterionVigneras}
Let $k$ be a $\mathbb{Z}[\frac{1}{p},\zeta_p]$-algebra which is a field. An irreducible $k$-representation $V$ of $G_n$ is cuspidal if and only if $V^{(n)}$ is one-dimensional and $V^{(i)}=0$ for $i=1,\dots,n-1$. 
\end{theorem}

Finally, we state some basic facts about how the spaces of bilinear forms interact with some of the Bernstein--Zelevinsky functors .

\begin{proposition}\label{bil spaces and bz functors} As usual, let $\one:G_{n+1}\to R^{\times}$ be the trivial character.
\begin{align} \label{eq BZ bil cancel Psi}
\bil_{P_{n+1}}(\Psi^+(V),\Psi^+(V'), \one)&\cong \bil_{G_n}(V,V',\one)\\ \label{eq BZ bil cancel Phi}
\bil_{P_{n+1}}(\Phi^+(V),\Phi^+(V'),\one)&\cong \bil_{P_n}(V,V',\one)\\
\label{eq BZ bil is zero}
\bil_{P_{n+1}}(\Psi^+(V),\Phi^+(V'),\one) &=0
\end{align}
In each statement above, $V$ and $V'$ are arbitrary representations living in the appropriate category.
\end{proposition}
\begin{proof}
This follows from \cite[Section 3.6]{bz77} and the adjunctions in \autoref{ss properties of BZ functors}.
\end{proof}

\begin{lemma} \label{lemma Gm to Pm}Let $V \in \Rep_R(P_n)$ and $V' \in \Rep_R(G_n)$. Then
    \[ \bil_{G_n}(\Phi^+V,V',\one)\cong\bil_{P_n}(V,V',\one). \]
\end{lemma}
\begin{proof}
The proof is the same as in \cite[Lemma 3.5]{matringe-moss} or \cite[Lemma 3.8]{kurinczuk-matringe}.
\end{proof} 

\subsection{Whittaker models} \label{whittakermodels_subsection}
Recall that we fixed a nontrivial character $\psi: \F_q \to \Z[\frac{1}{p},\zeta_p]^\times$ and its extension $\psi_R: \F_q \to R^\times$ in \autoref{ss derivative functors}. The \emph{Whittaker space} for $G_n$, or \emph{Gelfand--Graev representation} of $G_n$, is
    \[ \calw(\psi_R) := \Ind_{N_n}^{G_n} \psi_R \]
where $\psi_R$ is viewed as a character of $N_n$ via the map
\begin{align*}
     N_n \to N_n/[N_n,N_n] \xrightarrow{\sim} & (\F_q)^{\oplus n-1} \to \F_q \\
    & (y_1, \ldots, y_{n-1}) \mapsto y_1 + \cdots +y_{n-1}.
\end{align*}
Since we defined $\psi$ over the base ring,  $\calw(\psi_R)$ does not depend on the choice of $\psi$. See \cite[Remark 2.2]{matringe-moss} for a discussion of this.

\begin{definition}\label{WhittakerType}
We say that $(\rho,V) \in \Rep_R(G_n)$ is of \textit{$\psi$-Whittaker type} (or just Whittaker type) if the $n$-th derivative $V^{(n)}$ is a free $R$-module of rank $1$.
\end{definition}
\begin{remark}
We will sometimes call an irreducible representation of $\psi$-Whittaker type \emph{$\psi$-generic} or \emph{generic}. Without the irreducibility assumption, there is a distinction between Whittaker type and generic, as described in the next definition.
\end{remark}

By Frobenius reciprocity and \autoref{top derivative explicit} there is an isomorphism
    \[ \Hom_R(V^{(n)}, R) \xrightarrow{\sim} \Hom_{R[G_n]}(V, \calw(\psi_R)). \]

\begin{definition}
Suppose $(\rho,V)$ is of $\psi$-Whittaker type. Then the choice of a generator of $\Hom_R(V^{(n)},R)$ gives a map $V \to \calw(\psi_R)$.
\begin{enumerate}
    \item The image of $V \to \calw(\psi_R)$ is denoted $\calw(V, \psi_R)$ and is called the \textit{$\psi$-Whittaker model} (or just \textit{Whittaker model}) of $V$. Note the image does not depend on the choice of generator.
    \item We say that $V$ is \textit{essentially $\psi$-generic} if the map $V \to \calw(\psi_R)$ is injective. In this case $V$ and $\calw(V, \psi_R)$ are isomorphic as $R[G_n]$-modules.
\end{enumerate} 
\end{definition}

\begin{example}
For an example of a representation that is $\psi$-Whittaker type but not essentially $\psi$-generic, let $R$ be a field of characteristic $\ell$, let $V_1$ be an irreducible generic representation of $G_n$, let $V_2$ be any non-generic representation of $G_n$ (e.g. the trivial representation for $n \geq 2$), and take any extension $V$ of $V_2$ by $V_1$ (e.g. $V_1\oplus V_2$). By exactness of the derivative functor, $V^{(n)} = V_1^{(n)}\oplus V_2^{(n)} = V_1^{(n)} \cong R$, so $V$ is of $\psi$-Whittaker type. However, the map $V\to \calw(\psi_R)$ contains the subrepresentation $V_2$ in its kernel, so $V$ is not essentially $\psi$-generic.
\end{example}

\begin{lemma}\label{whitt_model_base_change}
Let $R\to R'$ be a homomorphism of rings. If $(\rho,V)$ is of $\psi$-Whittaker type, so is $(\rho\otimes_RR', V\otimes_RR')$ and 
$$\mathcal{W}(V\otimes_RR',\psi_{R'}) = \mathcal{W}(V,\psi_R)\otimes_RR'.$$
\end{lemma}
\begin{proof}
Since $\verts{N}$ is invertible in $R$ and $R'$, it follows from the existence of the projector in \autoref{coinvariants_projector} that $(V\otimes_RR')^{N,\psi_{R'}} = (V^{N,\psi_R})\otimes_RR'$ and hence also $(V\otimes_RR')_{N,\psi_{R'}} = (V_{N,\psi_R})\otimes_RR'$. This proves that $V\otimes_RR'$ is also of Whittaker type. Next, if $\lambda$ is a generator of the rank-one $R$-module $(V_{N,\psi_R})^{\vee}$, the Whittaker model of $V$ is
\begin{align*}
    V&\to \mathcal{W}(V,\psi_R)\\
    v&\mapsto W_v
\end{align*}
where $W_v(g) = \lambda(gv)$. In particular, $\lambda\otimes 1$ is generator of $((V\otimes_RR')_{N,\psi_{R'}})^{\vee}$ and the Whittaker model of $V\otimes_RR'$ is given by $$W_{v\otimes 1}(g) = (\lambda\otimes 1)(gv) = \lambda(gv)\otimes 1 = W_v(g)\otimes 1.$$ In particular, $\mathcal{W}(V\otimes_RR',\psi_{R'})=\mathcal{W}(V,\psi_R)\otimes_RR'$.
\end{proof}

The following Lemma is sometimes described as the existence of so-called ``Bessel vectors.''
\begin{lemma}\label{lem:restriction_to_pn}
If $(\rho,V)\in \Rep_R(G_n)$ is of $\psi$-Whittaker type, the map
\begin{align*}
\calw(V,\psi_R)&\to \Ind_{N_n}^{P_n}\psi_R\\
W&\mapsto \restr{W}{P_n}
\end{align*}
is surjective.
\end{lemma}
\begin{proof}
Denote $\calw(W,\psi_R)$ by $\calw$. We will exhibit a subspace of $\calw$ that maps isomorphically to $\Ind_{N_n}^{P_n}\psi_R$ under this map, namely it is the bottom step $(\Phi^+)^{n-1}\Psi^+(\calw^{(n)})$ (corresponding to $k=n$) of the filtration in \autoref{bzfiltration} applied to $\calw$.

By \autoref{coinvariants_projector} and \autoref{top derivative explicit}, the natural quotient map 
\begin{align*}
\calw &\to \calw^{(n)}\\
v&\mapsto \bar{v}
\end{align*}
maps $\calw^{N_n,\psi_R}$  isomorphically onto $\calw^{(n)}$. We view $\calw^{(n)}$ as the trivial representation of $G_0 = \{1\}$. The definition of $\Phi^+$ and transitivity of induction identifies $\ind_{N_n}^{P_n}\psi_R \cong (\Phi^+)^{n-1}\Psi^+(\calw^{(n)})$.

The inclusion $(\Phi^+)^{n-1}\Psi^+(\calw^{(n)}) \hookrightarrow \calw$ coming from \autoref{eq quotients from derivatives} corresponds to the aforementioned isomorphism $\calw^{(n)}\cong \calw^{N_n,\psi_R}$ under the following adjunctions: 
\begin{align*}
\Hom_R(\calw^{(n)},\calw^{N_n,\psi_R}) & \cong \Hom_{R[N_n]}(\psi_{\calw^{(n)}},\calw)\\
&\cong \Hom_{R[P_n]}(\ind_{N_n}^{P_n}\psi_{\calw^{(n)}},\calw).
\end{align*}
Let us be explicit. If $v$ is an $R$-generator of $\calw^{N_n,\psi_R}$, the function $f_{\bar{v}}$ supported on $N_n$ such that $f_{\bar{v}}(n) = \psi_R(n)\bar{v}$, $n\in N_n$, is a generator of $\Ind_{N_n}^{P_n}\psi_{\calw^{(n)}}$ (\cite[I.5.2]{vig_book}). The inclusion $\Ind_{N_n}^{P_n}\psi_{\calw^{(n)}}\hookrightarrow \calw$ sends $f_{\bar{v}}$ to $v$ (\cite[I.5.7]{vig_book}).

As $\calw$ is a subset of $\Ind_{N_n}^{G_n}\psi_R$, we will view elements of $\calw$ as functions on $G_n$. In this context, the value $w(g)$ of an element $w\in \calw$ is the element of $R$ corresponding to $\overline{gv}$ in our fixed isomorphism $\calw^{(n)}\cong R$.

Since our generator $v$ of $\calw^{N_n,\psi_R}$ satisfies $nv = \psi_R(n)v$ for $n\in N_n$, it follows that for $g\in G_{n-1}$, $\left(\begin{smallmatrix} I_{n-1} & u\\0&1\end{smallmatrix}\right)\in U_n$, we have
$$\psi_R\left(\begin{smallmatrix} I_{n-1} & u\\0&1\end{smallmatrix}\right)v\left(\begin{smallmatrix} g & 0\\0&1\end{smallmatrix}\right) = \psi_R\left(\begin{smallmatrix} I_{n-1} & g^{-1}u\\0&1\end{smallmatrix}\right)v\left(\begin{smallmatrix} g & 0\\0&1\end{smallmatrix}\right).$$ Since $P_{n-1}$ is the stabilizer of $\psi_R$ in $G_{n-1}$, it follows that the support of $v|_{G_{n-1}}$ is contained in $P_{n-1}$. But the same argument with $g\in G_{n-2}$ and $u\in U_{n-1}$ shows that~$v|_{G_{n-2}}$ is supported on $P_{n-2}$. Repeating this, we conclude that the restriction of $v$ to~$P_n$ is supported only on $N_n$. Since the values of $v$ and $f_{\bar{v}}$ agree on $N_n$ by construction, we conclude that $v\mid_{P_n} = f_{\bar{v}}$. Since $f_{\bar{v}}$ is a generator of $\Ind_{N_n}^{P_n}\psi_R$, we conclude.
\end{proof}

Let $W: G_n \to R$ be an element of $\calw(V,\psi_R)$ and let $\widetilde{W}$ be the function defined by
    \[ \widetilde{W}(g) = W(w_n (^\iota g)), \]
where $w_n$ is defined to be the antidiagonal matrix in $G_n$ with $1$'s along the antidiagonal, and ${^\iota}g := {^t}g^{-1}$. Then $\widetilde{W}(ng) = W(w_n({^\iota}n) ({^\iota}g)) = \psi_R^{-1}(n)W(w_n({^\iota}g)) = \psi_R^{-1}(n)\widetilde{W}(g)$ for all $n \in N_n$ and so $\widetilde{W}$ defines an element of $\calw({^\iota}V,\psi_R^{-1})$, where ${^\iota}V$ denotes the representation given by precomposing $V$ with the involution ${^\iota}$.

\subsection{Exceptional representations}

Later when defining gamma factors for pairs of representations we will need to exclude certain exceptional pairs. The term ``exceptional'' follows \cite[Section 17]{ps}, which studies representations of $\GL_2(\F_q)$ on $\C$-vector spaces and defines the notion of exceptional for \textit{characters}. Our definition is a higher-dimensional generalization of \textit{op. cit}.

\begin{definition} \label{def exceptional}
If $(\pi,V) \in \Rep_R(G_n)$ and $(\pi',V') \in \Rep_R(G_m)$ we say that $(V,V')$ is an \textit{exceptional pair}, or that $V'$ is \emph{exceptional for $V$} (or vice versa) if there exists an integer $t \in \{1,\dots,\min(m,n) \}$ such that
\[ \bil_{G_t}(\calw(V,\psi_R)^{(n-t)}, \calw(V',\psi_R^{-1})^{(m-t)}, \one)\neq\{0\}. \]
\end{definition}

We remark that the notion of exceptional pair only depends on the Whittaker models of the representations.

\section{Functional equation}
\label{section functional equation}

Fix $(\pi,V) \in \Rep_R(G_n)$ and $(\pi',V') \in \Rep_R(G_m)$ both of Whittaker type. Assume that~$\pi'$ is not exceptional for $\pi$. In this section we construct a gamma factor $\gamma(\pi \times \pi', \psi_R)$ for the pair $(\pi,\pi')$. Since this will only depend on the Whittaker models, we make the following abbreviations to ease the notation in this section:
    \[ \calw := \calw(V,\psi_R) \text{ and }
    \calw' := \calw(V',\psi_R^{-1}). \]

\subsection{Gamma factor and functional equation when \texorpdfstring{$n > m$}{n>m}}\label{sec:ngm}

We first suppose $n > m$; the $n = m$ case is slightly different, so we address it afterwards.

Recall the subgroup of $G_n$ given by
    \[ U_{n, n-m-1}:=\left\{\left(\begin{smallmatrix} I_{m+1} & z\\ & y\end{smallmatrix}\right)\ :\ z\in \text{Mat}_{m+1, n-m-1}(\mathbb{F}_q),\ y\in N_{n-m-1}\right\}. \]
Inflate $\calw'$ to an $R[G_mU_{n,n-m-1}]$-module by letting $U_{n,n-m-1}$ act trivially.

Consider the following finite field analogue of the integral defined in \cite[Section 2.4]{rankin-selberg-convolutions}. If $W: G_n \to R$ and $W': G_m \to R$ are two functions and $j \in \set{0,\dots,n-m-1}$ then let
\begin{equation}\label{pairing}
I(W,W';j) := \sum_{g\in N_m\backslash G_m}\sum_{y\in \mat_{j\times m}}W\paren{\begin{pmatrix} g&0&0\\y&I_j&0\\0&0&I_{n-m-j}\end{pmatrix}}W'(g).
\end{equation}

If we let
    \[ w_{n,m}:=\begin{pmatrix}I_m & \\ & w_{n-m}\end{pmatrix} \]
then a direct computation (done in detail in \cite[Lemma 5.2]{roditty}) shows that the maps
\begin{align*}
    (W,W') &\mapsto I(W,W';0) = \sum_{g\in N_m\backslash G_m}W\left(\begin{smallmatrix}g & 0\\ 0 & I_{n-m}\end{smallmatrix}\right)W'(g) \\
    (W,W') &\mapsto I(w_{n,m}\wt{W},\wt{W'};n-m-1) = \sum_{g\in N_m\backslash G_m}\sum_{y\in \mat_{n-m-1 \times m}}W\left(\begin{smallmatrix}0 & 1 &0\\
 0&0 & I_{n-m-1}\\ g &0 & y\end{smallmatrix} \right)W'(g)
\end{align*}
define elements of
    \[ \bil_{G_mU_{n,n-m-1}}(\calw, \calw',\one\otimes \psi_R), \]
where $\one \otimes \psi_R$ is the character acting trivially on $G_m$ and by $\psi_R$ on $U_{n,n-m-1}$.

In this section we use the calculus of the Bernstein--Zelevinsky functors to analyze this space of bilinear forms. Our main result is the following.

\begin{theorem}
\label{thm:bil_space_one_dimensional}
The space
    \[ \bil_{G_mU_{n,n-m-1}}(\calw, \calw',\one\otimes \psi_R) \]
is free of rank one over $R$ generated by $I(W,W';0)$.
\end{theorem}

As a corollary, we deduce the functional equation which defines the gamma factor $\gamma(\pi \times \pi', \psi_R)$. 

\begin{corollary}
\label{cor:functional_equation}
There exists a unique element $\gamma(\pi\times \pi', \psi_R) \in R$ such that
    \[ I(W,W';0)\gamma(\pi\times \pi', \psi_R) = I(w_{n,m}\wt{W},\wt{W'};n-m-1) \]
for all $W \in \calw$ and $W' \in \calw'$.
\end{corollary}

\begin{remark}
In the next section we prove a more general functional equation and use it to deduce that in fact $\gamma(\pi \times \pi', \psi_R) \in R^\times$, see \autoref{cor:gamma_invertible}.
\end{remark}

\begin{corollary}\label{compatibility_with_specialization}
If $f:R\to R'$ be a ring homomorphism, then
    \[ f(\gamma(\pi \times \pi',\psi_R)) = \gamma(\pi\otimes_{R}R'\times \pi'\otimes_{R}R',\psi_{R'}). \]
\end{corollary}
\begin{proof}
By applying $f$ to both sides of the functional equation in \autoref{cor:functional_equation} and using \autoref{whitt_model_base_change}, we find that $f(\gamma(\pi\otimes \pi',\psi_R))$ satisfies the same functional equation as $\gamma(\pi\otimes_{R}R'\times \pi\otimes_{R}R',\psi_{R'})$. Therefore the uniqueness in \autoref{cor:functional_equation} implies they are equal.
\end{proof}

\begin{remark} 
Note that if $V$ is irreducible cuspidal, there are no representations that are exceptional for $V$, by \autoref{CuspidalCriterionVigneras}. Thus, in this case, we recover the functional equation in the special cases treated in \cite{roditty,nien}.
\end{remark}

The rest of this section is devoted to the proof of \autoref{thm:bil_space_one_dimensional}. Our strategy follows that of \cite[3.2]{kurinczuk-matringe} and \cite[3.2]{matringe-moss} in the setting of $p$-adic groups but there is a key lemma in the $p$-adic setting which completely fails in the setting of finite groups for lack of unramified characters, namely \cite[Lemma 3.6]{matringe-moss}. This failure is precisely what necessitates the exclusion of the exceptional representations for $V$ in \autoref{thm:bil_space_one_dimensional}. Without the exclusion of exceptional characters the theorem is false, c.f. \cite[Lemma 4.1.4, Theorem 4.3.3]{roditty}.  

Our main tool will be the properties of the Bernstein--Zelevinsky functors established in \autoref{bil spaces and bz functors} and \autoref{lemma Gm to Pm}. The proof of \autoref{thm:bil_space_one_dimensional} proceeds by several reductions steps, which we state as lemmas.

\begin{lemma} There is a canonical isomorphism
    \[ \bil_{G_mU_{n,n-m-1}}(\calw,\calw',\one\otimes \psi_R) \xrightarrow\sim \bil_{G_m}((\Phi^-)^{n-m-1}\calw, \calw', \one). \]
\end{lemma}
\begin{proof}
By definition, $\bil_{G_mU_{n,n-m-1}}(\calw,\calw',\one\otimes \psi_R)$ is
\begin{equation}\label{eq reduction 1 hom space}
    \hom_{R[G_mU_{n,n-m-1}]}(\calw\otimes \calw',\one\otimes \psi_R),
\end{equation}
where $G_mU_{n,n-m-1}$ acts diagonally on $\calw \otimes \calw'$. But any such homomorphism must factor through $\tau \otimes \calw'$ where $\tau$ is the quotient of $\calw$ by the submodule generated by elements of the form $uW - \psi_R(u)W$ for $u\in U_{n,n-m-1}$, $W\in \calw$. Moreover, this quotient is universal for this property, so \autoref{eq reduction 1 hom space} is isomorphic to 
    \[ \hom_{R[G_m]}(\tau\otimes \calw',\one). \]
Now the result follows from the fact that $\tau = (\Phi^-)^{n-m-1}\calw$, see \autoref{top derivative explicit}. 
\end{proof}

We now consider the Bernstein--Zelevinsky filtration of $\calw$ given by \autoref{bzfiltration}. After applying $(\Phi^-)^{n-m-1}$ to the filtration we have
    \[ 0\subset (\Phi^-)^{n-m-1}\calw_n\subset  \cdots \subset (\Phi^-)^{n-m-1}\calw_1 = (\Phi^-)^{n-m-1}\calw, \]
which is now a filtration of representations of $P_{m+1}$. Following \autoref{eq quotients from derivatives} and exactness of~$\Phi^-$, the successive quotients are given by 
$$(\Phi^-)^{n-m-1}(\calw_k/\calw_{k+1}) = (\Phi^-)^{n-m-1}(\Phi^+)^{k-1}\Psi^+(\calw^{(k)}).$$ Note that since $\calw^{(n)} = V^{(n)} = \one$ by assumption, the identity $\Phi^-\Phi^+\cong \id$ implies that the bottom step of the filtration is the submodule 
$$(\Phi^-)^{n-m-1}(\Phi^+)^{n-1}\Psi^+(\one) = (\Phi^+)^{m}\Psi^+(\one)\subset (\Phi^-)^{n-m-1}\calw.$$

\begin{lemma} \label{lemma bottom step 1}
The restriction map
\begin{align*}
    \bil_{G_m}((\Phi^-)^{n-m-1}\calw,\calw',\one) &\to \bil_{G_m}((\Phi^+)^m\Psi^+(\one), \calw',\one) \\
    B &\mapsto \restr{B}{(\Phi^+)^{m}\Psi^+(\one)\times \calw'}
\end{align*}
is injective.
\end{lemma}
\begin{proof}
If $B|_{(\Phi^+)^{m}\Psi^+(\one)\times \calw'} = 0$, it defines a bilinear form on the next quotient
    \[ (\Phi^-)^{n-m-1}(\Phi^+)^{n-2}\Psi^+(\calw^{(n-1)})\times \calw'. \]
In fact, we will show that the spaces of bilinear forms on each successive quotient,
$$\bil_{G_m}((\Phi^-)^{n-m-1}(\Phi^+)^{i}\Psi^+(\calw^{(i+1)}), \calw',\one),$$
are identically zero for $i=0,\dots, n-2$. We will consider three cases.\\

\noindent\textbf{Case 1:} $i <n-m-1$.

The module $(\Phi^-)^{n-m-1}(\Phi^+)^{i}\Psi^+(\calw^{(i+1)})$ is zero since $\Phi^-\Psi^+\cong 0$ and $\Phi^{-} \Phi^+ \cong \text{id}$, see \autoref{ss properties of BZ functors}.  \\

\noindent\textbf{Case 2:} $i= n-m-1$.

We have 
\begin{align*}
\bil_{G_m}((\Phi^-)^{n-m-1}(\Phi^+)^{i}\Psi^+(\calw^{(i+1)}), \calw',\one) 
&= \bil_{G_m}(\Psi^+(\calw^{(n-m)}), \calw',\one)\\
&= \bil_{G_m}(\calw^{(n-m)}, \calw', \one)\\
&= \{0\},
\end{align*}
where the last equality is from the non-exceptional assumption.\\

\noindent\textbf{Case 3:} $i>n-m-1$.

In this case, we are considering the space
$$\bil_{G_m}((\Phi^+)^{i-(n-m-1)}\Psi^+(\calw^{(i+1)}), \calw',\one).$$ To keep things tidy, we introduce a new index:
$t := n-i-1,$ so that 
\begin{align*}
i-(n-m-1) &= m - t\\
i+1 &= n-t.
\end{align*}
Because $n-m\leq i\leq n-2$ in the present case, the range of $t$ is
$1\leq t \leq m-1.$ Our goal is to prove
$$\bil_{G_m}((\Phi^+)^{m-t}\Psi^+(\calw^{(n-t)}), \calw',\one)=\{0\}.$$

First, we can restrict to $P_m$  following \autoref{lemma Gm to Pm}, 
    \[ \bil_{G_m}((\Phi^+)^{m-t}\Psi^+(\calw^{(n-t)}), \calw',\one)=\bil_{P_m}((\Phi^+)^{m-t-1}\Psi^+(\calw^{(n-t)}),\calw',\one). \]
As a representation of $P_m$, we filter $\calw'$ using \autoref{bzfiltration}: the successive quotients in the filtration are
$(\Phi^+)^{m-t'-1}\Psi^+((\calw')^{(m-t')})$ with $0\leq t'\leq m-1$.

At the bottom of the filtration, where $t'=0$, our bilinear forms restrict to elements of
    \[ \bil_{P_m}((\Phi^+)^{m-t-1}\Psi^+(\calw^{(n-t)}),(\Phi^+)^{m-1}\Psi^+((\calw')^{(m)}),\one),\]
which equals zero by \autoref{eq BZ bil cancel Phi} and \autoref{eq BZ bil is zero} since $t>0$. Similarly, when a bilinear form is restricted to any step in the filtration where $t\neq t'$, the same argument gives
$$\bil_{P_m}((\Phi^+)^{m-t-1}\Psi^+(\calw^{(n-t)}),(\Phi^+)^{m-t'-1}\Psi^+((\calw')^{(m-t')}),\one)=\{0\}.$$

Thus it remains only to treat the case where $t = t'$, where
$$\bil_{P_m}((\Phi^+)^{m-t-1}\Psi^+(\calw^{(n-t)}),(\Phi^+)^{m-t-1}\Psi^+((\calw')^{(m-t)}),\one)=\{0\},$$ by the assumption that $V'$ is non-exceptional for $V$.
\end{proof}

\begin{lemma}\label{concluding_lemma}
\begin{align*}
    \bil_{P_m}((\Phi^+)^{m-1}\Psi^+(\one), \calw', \one) &\hookrightarrow \bil_{P_m}((\Phi^+)^{m-1}\Psi^+(\one),(\Phi^+)^{m-1}\Psi^+(\one),\one)\\
    &= \bil_{G_0}(\one,\one,\one)\\
    &= R
\end{align*}
\end{lemma}
\begin{proof}
First, we note that the second isomorphism is given by \autoref{eq BZ bil cancel Psi} and \autoref{eq BZ bil cancel Phi} of \autoref{bil spaces and bz functors}, and the third isomorphism is trivial.

Next, we will consider the injection on the first line. Consider the filtration of $\calw'$ as in \autoref{bzfiltration}. From \autoref{eq quotients from derivatives}, the bottom step of the filtration is~$(\Phi^+)^{m-1}\Psi^+(\one)$. The injection on the first line of the lemma is given by restricting a bilinear form $B$ to this bottom step in the second factor. We will prove that this restriction map is injective. Assume a bilinear form is zero when restricted to
    \[ (\Phi^+)^{m-1}\Psi^+(\one)\times (\Phi^+)^{m-1}\Psi^+(\one). \]
Then it defines a bilinear form in
    \[ \bil_{P_m}((\Phi^+)^{m-1}\Psi^+(\one),(\Phi^+)^{i-1}\Psi^+((\calw')^{(i)}),\one) \]
for an integer $i<m$. But this space is zero by the same argument as in the proof of \autoref{lemma bottom step 1}, thanks to \autoref{eq BZ bil cancel Phi} and \autoref{eq BZ bil is zero} of \autoref{bil spaces and bz functors}. Hence $B=0$, and the injectivity is proved.
\end{proof}

Finally, we use the following fact to put everything together.

\begin{theorem}[{{\cite{Orzech}, \cite{OrzechMO}, \cite{Grinberg}}}]\label{Orzech}
Suppose $A$ is a commutative ring, $M$ is a finitely generated $A$-module and $N \subset M$ is an $A$-submodule. Then any surjection $f: N \twoheadrightarrow M$ is an isomorphism.
\end{theorem}

\begin{proof}[Proof of \autoref{thm:bil_space_one_dimensional}]
The above three lemmas give us an injection
    \[ \bil_{G_m U_{n,n-m-1}}(\calw, \calw', \one \otimes \psi_R) \hookrightarrow R \]
By \autoref{Orzech} it suffices to find $W \in \calw$ and $W' \in \calw'$ such that $I(W,W';0) = 1$, because then the evaluation map
    \[ \ev_{W,W'}: \bil_{G_mU_{n,n-m-1}}(\calw,\calw',\one \otimes \psi_R) \twoheadrightarrow R \]
sends $I(-,-;0)$ to a unit and is therefore surjective. 

By \autoref{top derivative explicit}, the map $\lambda$ defined by $V'\to V'_{N_m,\psi_R^{-1}}\xrightarrow{\sim}R$ is surjective. Recall that Frobenius reciprocity associates to $\lambda$ a map $V'\to \calw'$, which is defined by $v\mapsto W_v$ where $W_v(g):= \lambda(gv)$. There is a natural ``evaluation at the identity'' map $\text{ev}_1:\calw'\to R$ given by
\begin{align*}
\text{ev}_1: \mathcal{W'} & \to R\\
W_v&\mapsto W_v(1).
\end{align*}
Since $W_v(1) = \lambda(v)$, the surjectivity of $\lambda$ implies there exists $W_v\in \calw'$ such that $W_v(1) = 1$. Let $W'$ be any such choice of $W_v$. 

Given an arbitrary element $\phi$ of $\ind_{N_n}^{P_n}\psi_R$,  \autoref{lem:restriction_to_pn} tells us there exists $W$ in $\calw$ such that $W|_{P_n}=\phi$. Note that when we evaluate the sum defining $I(W,W';0)$ we only ever evaluate $W$ on elements of $P_n$, so we may choose $\phi$ so it is supported only on $N_n$ and such that $\phi(1)=1$. Now for any choice of $W\in \calw$ restricting to $\phi$, we obtain $I(W,W';0)=1$.
\end{proof}

\begin{remark}
Note that if $R$ is a field, this final surjectivity argument is unnecessary because any nonzero bilinear form (e.g. $I(-,-;0)$) will provide a basis vector.
\end{remark}

\subsection{More general functional equation when \texorpdfstring{$n > m$}{n>m}}
In this subsection we use \autoref{cor:functional_equation} to deduce a slightly more general functional equation for the gamma factor. First we introduce some notation. Assume the same notation from the previous section.

\begin{corollary}
\label{cor:more_general_functional_eqn}
Let $j$ be an integer, $0\leq j\leq n-m-1$. In the same setup as \autoref{cor:functional_equation}, we have
    \[ I(W,W';j)\gamma(\pi \times \pi',\psi_R) = I(w_{n,m}\widetilde{W},\widetilde{W'};k), \]
where $k = n-m-1-j$.
\end{corollary}
\begin{proof}
The same argument as in \cite[Thm 5.4]{roditty} works here.
\end{proof}

\begin{corollary}\label{cor:gamma_invertible}
In the same setup as \autoref{cor:functional_equation}, the element $\gamma(\pi\times\pi',\psi_R)$ is invertible in $R$.
\end{corollary}
\begin{proof}
One approach would be to prove that $I(w_{n,m}\wt{W}',\wt{W'};n-m-1)$ is also a generator of $\bil_{G_mU_{n,n-m-1}}(\calw, \calw',\one\otimes \psi_R)$, but we will instead use 
\autoref{cor:more_general_functional_eqn}. Since 
$w_{n,m}\widetilde{W}$ defines an element of $\mathcal{W}({^\iota V},\psi_R^{-1})$, the functional equation gives
\begin{align*}
    I(W,W';0)\gamma(\pi\times \pi',\psi_R)\gamma({^\iota \pi}\times {^\iota \pi'},\psi_R^{-1}) & = I(w_{n,m}\wt{W},\wt{W'};n-m-1)\gamma({^\iota \pi}\times {^{\iota}\pi'},\psi_R^{-1})\\
    & = I(w_{n,m}\widetilde{w_{n,m}\widetilde{W}},\wt{\wt{W'}};0)\\
    &= I(W, W';0)\ .
\end{align*}
Thus it's enough to show the existence of $W$ and $W'$ such that $I(W,W';0) = 1$, which is done in the proof of \autoref{thm:bil_space_one_dimensional}.
\end{proof}

\subsection{Gamma factor and functional equation when \texorpdfstring{$n = m$}{n=m}}

Now we address the case when $n = m$.

Let $C(\F_q^n,R)$ denote the set of all functions $\Phi: \F_q^n \to R$. Since $G_n$ naturally acts (on the right) on $\F_q^n$, the set $C(\F_q^n,R)$ acquires an $R$-linear left $G_n$-action by setting
    \[ (g \cdot f)(x) = f(x \cdot g). \]
The $R$-subspace
    \[ C_0(\F_q^n,R) = \set{f \in C(\F_q^n,R) : f(0,\dots,0) = 0} \]
is $G_n$-stable.

In order to formulate a functional equation, we define trilinear forms instead of bilinear forms to take into account the functions in $C(\F_q^n,R)$. If $W, W': G_n \to R$ are two functions and $\Phi \in C(\F_q^n, R)$ then let
    \[ I(W,W',\Phi) := \sum_{g \in N_n \backslash G_n} W(g) W'(g) \Phi(\eta g) \]
where $\eta = \begin{pmatrix}0 & \cdots & 0 & 1\end{pmatrix}$. For $\Phi \in C(\F_q^n, R)$ let $\wh\Phi \in C(\F_q^n,R)$ denote the Fourier transform
    \[ \wh\Phi(a) = \sum_{x \in \F_q^n} \Phi(x)\psi_R(\langle a, x \rangle) , \]
where $\langle - , - \rangle$ denotes the standard inner product on $\F_q^n$, given for $a = (a_1, \ldots, a_n)$ and $x = (x_1, \ldots, x_n)$ by $\langle a , x \rangle = a_1 x_1 + \cdots + a_n x_n$.

The maps
\begin{align*}
    (W,W',\Phi) &\mapsto I(W,W',\Phi) \\
    (W,W',\Phi) &\mapsto I(\wt{W},\wt{W'},\wh{\Phi})
\end{align*}
define elements of
    \[ \tril_{G_n}(\calw, \calw', C(\F_q^n,R)). \]

\begin{theorem}\label{thm:tril_space_one_dimensional}
If $(V,V')$ is not an exceptional pair then
    \[ \tril_{G_n}(\calw, \calw',C(\F_q^n,R)) \]
is a free $R$-module of rank $1$ generated by $I(W,W',\Phi)$. 
\end{theorem}
\begin{proof}
We closely follow \cite[Proposition 3.7]{kurinczuk-matringe}. The $G_n$-equivariant exact sequence of $R$-modules
    \[ 0 \to C_0(\F_q^n,R) \to C(\F_q^n,R) \to \one \to 0 \]
consists entirely of free finite rank $R$-modules and thus splits, so
    \[ 0 \to \calw \otimes_R \calw' \otimes_R C_0(\F_q^n,R) \to \calw \otimes_R \calw' \otimes_R C(\F_q^n,R) \to \calw \otimes_R \calw' \to 0 \]
is still a $G_n$-equivariant exact sequence. Since $(V,V')$ is not an exceptional pair we see that $\bil_{G_n}(\calw, \calw', \one) = 0$. So in view of the above sequence and the left-exactness of the Hom functor, we see that $\tril_{G_n}(\calw, \calw', C(\F_q^n,R))$ injects into $\tril_{G_n}(\calw, \calw', C_0(\F_q^n, R))$.

Note that $C_0(\F_q^n,R)$ is isomorphic as a $G_n$-representation to $\Ind_{P_n}^{G_n} \one$ because the orbit of the vector $\eta = (0,\dots,0,1)$ under the standard right action of $G_n$ on $\F_q^n$ is $\F_q^n - \set{(0,\dots,0)}$ and the stabilizer is $P_n$. Again, using that
induction commutes with taking duals,
\begin{align*}
    \tril_{G_n}(\calw, \calw', C_0(\F_q^n,R)) &= \Hom_{R[G_n]}(\calw \otimes_R \calw' \otimes_R \Ind_{P_n}^{G_n} \one, \one) \\
    &= \Hom_{R[G_n]}(\calw \otimes_R \calw', (\Ind_{P_n}^{G_n} \one)^\vee) \\
    &= \Hom_{R[G_n]}(\calw \otimes_R \calw', \Ind_{P_n}^{G_n} \one) \\
    &= \bil_{P_n}(\calw, \calw', \one).
\end{align*}

Recall from above in \autoref{bzfiltration} that $\calw$ admits a filtration of length $n$ by $P_n$-subrepresentations with successive quotients isomorphic to $(\Phi^+)^{k-1}\Psi^+(\calw^{(k)})$ for $k = 1,\dots,n$, and the same is true for $\calw'$. But in view of \autoref{bil spaces and bz functors}
    \[ \bil_{P_n}((\Phi^+)^{k-1}\Psi^+(\calw^{(k)}), (\Phi^+)^{j-1}\Psi^+((\calw')^{(j)}), \one) \]
is zero unless $k = j$, in which case it's equal to
    \[ \bil_{P_n}((\Phi^+)^{k-1}\Psi^+(\calw^{(k)}), (\Phi^+)^{k-1}\Psi^+((\calw')^{(k)}), \one) = \bil_{G_{n-k}}(\calw^{(k)}, (\calw')^{(k)}, \one). \]
But $(V,V')$ is not an exceptional pair, so this vanishes for $k = 1, \dots, n-1$. The only surviving piece, then, is when $k = j = n$ and so using \autoref{bil spaces and bz functors} we see that there is an injection
\begin{align*}
    \bil_{P_n}(\calw, \calw', \one) &\hookrightarrow \bil_{P_n}((\Phi^+)^{n-1}\Psi^+(\one), (\Phi^+)^{n-1}\Psi^+(\one),\one) \\
    &= \bil_R(\one,\one,\one) \\
    &= R
\end{align*}
We have therefore found an $R$-module injection
    \[ \tril_{G_n}(\calw, \calw', C(\F_q^n, R)) \hookrightarrow R \]
By \autoref{Orzech} it suffices to find $W \in \calw$ and $W' \in \calw'$ such that $I(W,W',\delta_\eta) = 1$, where $\delta_\eta (x)$ equals $1$ if $x = \eta$ and equals $0$ otherwise. This is because then the evaluation map
    \[ \ev_{W,W',\delta_\eta}: \tril_{G_n}(\calw,\calw',C(\F_q^n,R)) \twoheadrightarrow R \]
sends $I(W,W',\Phi)$ to a unit and is therefore surjective.

As in the proof of \autoref{thm:bil_space_one_dimensional} we can pick Whittaker functions $W \in \calw$ and $W' \in \calw'$ such that $W(1) = 1$, the restriction $W|_{P_n}$ is supported on $N_n$, and $W'(1) = 1$. Then $I(W,W',\delta_\eta) = 1$.
\end{proof}

\begin{corollary}
There exists a unique element $\gamma(\pi\times\pi', \psi_R)$ of $R^{\times}$ such that
    \[ I(W,W',\Phi)\gamma(\pi \times \pi', \psi_R) = I(\wt{W},\wt{W'},\wh{\Phi}) \]
for all $W \in \calw$ and $W' \in \calw'$.
\end{corollary}
\begin{proof}
\autoref{thm:tril_space_one_dimensional} shows that there exists such a $\gamma(\pi\times\pi',\psi_R) \in R$, so we need to show that it is a unit.

As in \autoref{cor:gamma_invertible}, we have
\begin{align*}
    I(W,W',\Phi)\gamma(\pi \times \pi',\psi_R)\gamma({^\iota}\pi \times {^\iota}\pi', \psi_R^{-1}) &= I(\wt{W},\wt{W'},\wh\Phi)\gamma({^\iota}\pi \times {^\iota}\pi', \psi_R^{-1}) \\
    &= I(\wt{\wt{W}},\wt{\wt{W'}},\wh{\wh{\Phi}}) \\
    &= I(W,W',\Phi)
\end{align*}
and the proof of \autoref{thm:tril_space_one_dimensional} gives us $W,W',\Phi$ such that $I(W,W',\Phi) = 1$.
\end{proof}

\section{Converse theorem}

Let $k = \overline{\mathbb{F}}_{\ell}$. In this section we prove a converse theorem for irreducible cuspidal $k$-representions, in which gamma factors take values in Artinian $k$-algebras.
\subsection{Projective envelopes} \label{subsec: projective envelopes}
Recall that $N_n$ denotes the subgroup of unipotent upper triangular matrices. Since the order of $N_n$ is relatively prime to $\ell$, the character~$\psi_k: N_n \to k^{\times}$ is a projective $k[N_n]$-module. Since $\Ind_{N_n}^{G_n}$ is left-adjoint to an exact functor, it takes projective objects to projective objects and therefore $\Ind_{N_n}^{G_n} \psi_k$ is a projective $k[G_n]$-module. We can then decompose $\Ind_{N_n}^{G_n}\psi_k$ as a direct sum
    \[ \Ind_{N_n}^{G_n} \psi_k = P_1^{\oplus e_1}\oplus \dots \oplus P_r^{\oplus e_r}, \]
where each $P_i$ is indecomposable and projective, and $P_i\not\cong P_j$ for $i\neq j$. However, we know that $\End_{G_n}(\Ind_{N_n}^{G_n} \psi_k)$ is a commutative ring (see \cites[Theorem 49]{steinberg}[Corollary 4.8]{zhang}{allen-wake-zhang}), so $e_i=1$ for all $i$. The commutativity of $\End_{G_n}(\Ind_{N_n}^{G_n} \psi_k)$ also implies that $\Hom_{G_n}(P_i,P_j) = 0$ when $i\neq j$.

There is a bijection \cite[I.A.7]{vig_book} between isomorphism classes of irreducible representations of $G_n$ and isomorphism classes of indecomposable projective $k[G_n]$-modules:

\begin{align*}
\{\text{irreducible $k[G_n]$-modules}\} &\leftrightarrow \{\text{indecomposable projective $k[G_n]$-modules}\}\\
\pi&\mapsto P(\pi)\\
\soc(P) &\mapsfrom P,
\end{align*}
where $P(\pi)$ denotes the projective envelope of $\pi$ and $\soc(P)$ denotes the socle (i.e. the largest semisimple subrepresentation) of $P$. Note also that, by duality, $\pi$ also occurs as a quotient of $P(\pi)$ and is in fact the only irreducible quotient of $P(\pi)$ \cite[Chapter~14]{SerreLinear}. In other words, $\pi$ is not only the socle of $P(\pi)$ but also its cosocle (i.e. the largest semisimple quotient).

Since $P_i$ is not isomorphic to $P_j$ for $i\neq j$, the bijectivity above implies that $\soc(P_i)$ is not isomorphic to $\soc(P_j)$. On the other hand, being contained in $\Ind_{N_n}^{G_n} \psi_k$, each $\soc(P_i)$ is irreducible and generic, and every irreducible generic representation must occur as a (the) submodule of some $P_i$. Thus in restricting to generic objects we have a bijection of isomorphism classes:

\begin{align*}
\{\text{irreducible \emph{generic} $k[G_n]$-modules}\} &\leftrightarrow \{P_1,P_2,\dots,P_r\}\\
\pi&\mapsto P(\pi)\\
\soc(P) &\mapsfrom P.
\end{align*}

Let $P := P(\pi)$ for an irreducible generic representation $\pi$. Note that $P$ is finite-dimensional as a $k$-vector space and that it has finite length as a $k[G_n]$-module. Since it is moreover indecomposable, we conclude that $R(\pi):=\End_{k[G_n]}(P)$ is local by Fitting's lemma. The ring $R(\pi)$ is a finite-dimensional commutative $k$-algebra because it is contained in $\End_{k[G_n]}(\Ind_{N_n}^{G_n} \psi)$.

\subsection{Duality and derivative of \texorpdfstring{$P(\pi)$}{P(pi)}}
Recall that we let $^\iota g := {^t g^{-1}}$ and that for any representation $(\pi,V)$ of $G_n$, we let $({^\iota \pi}, V)$ denote the representation ${^\iota \pi}(g)v := \pi({^\iota g})v$. 

\begin{lemma}
If $\pi: G_n \to \mathrm{Aut}(V)$ is an irreducible representation, one has ${^\iota \pi}\cong\pi^{\vee}$ where $\pi^{\vee}$ denotes the dual to $\pi$. 
\end{lemma}
\begin{proof}
   Following \cite[Chapter 18]{SerreLinear}, it suffices to show that $\pi^\vee$ and ${^\iota}\pi$ have the same Brauer character. Let $g \in G_n$ have order coprime to $\ell$.  Then \[ \tr\pi^{\vee}(g) = \tr {^t}\pi(g^{-1}) = \tr \pi(g^{-1}) = \tr \pi({^\iota}g) = \tr{^\iota}\pi(g) \] since every matrix in $G_n$ is conjugate to its transpose. 
\end{proof}

Now let us consider the representation $^{\iota}(\Ind_{N_n}^{G_n}\psi_k)$. Recall from
\autoref{whittakermodels_subsection} that $w_n$ denotes the antidiagonal matrix with $1$'s along the antidiagonal, and note that for $u \in N_n$, $\psi_k(w_n({^\iota u}){w_n^{-1}}) = \psi_k^{-1}(u)$. Also recall that for $W:G_n\to k$ an element of $\Ind_{N_n}^{G_n}\psi_k$, we let $\widetilde{W}$ be  the function defined by
$$\widetilde{W}(g) = W(w_n ({^\iota g})).$$ This function defines an element of $\Ind_{N_n}^{G_n}\psi_k^{-1}$ since for $u \in N_n$ we have 
\[ \widetilde{W}(ug) = W(w_n({^\iota u})({^\iota g})) = \psi_k^{-1}(u)W(w_n({^\iota g)}) = \psi_k^{-1}(u)\widetilde{W}(g).\]

For $h\in G_n$, the map
\begin{align*}
\Ind_{N_n}^{G_n}\psi_k &\to \Ind_{N_n}^{G_n}\psi_k^{-1}\\
W&\mapsto \widetilde{W}
\end{align*}
satisfies $\widetilde{(hW)} = {^\iota h}\widetilde{W}$, so the map is a $G_n$-equivariant isomorphism when the target is equipped with the $G_n$-action obtained by composing the right-translation action with the involution $g\mapsto {^\iota g}$.

Recall the notation from \autoref{ss derivative functors}: for $R$ a $k$-algebra and $P$ an $R[G_n]$-module:
$$P^{(n)} = P_{N_n,\psi_R} := P/P(N_n,\psi_R),$$ where $P(N_n,\psi_R)$ is the $R$-module generated by $uv-\psi_R(u)v$, $u\in N_n$, $v\in P$. Thus $P^{(n)}$ is the $(N_n,\psi_R)$-coinvariants, i.e. the largest quotient in the category of $R[N_n]$-modules on which $N_n$ acts via the character $\psi_R$.

Note that $P(N_n,\psi_k)$ is equal to the \emph{$k$-vector space} generated by the set $\{uv-\psi_k(u)v:u\in N_n, v\in P\}$, so $P^{(n)}$ is also the largest quotient on which $N_n$ acts via $\psi_k$ in the category of $k[N_n]$-modules. 

\begin{proposition}
\label{prop:Pderivative}
Let $\pi$ be an irreducible generic $k[G_n]$-module and let $P = P(\pi)$ be its corresponding indecomposable projective module, considered as a module over the ring $R =R(\pi)= \End_{k[G_n]}(P)$. Then $P^{(n)}$
is free of rank one as an $R$-module.
\end{proposition}
\begin{proof}
By \autoref{section_representations}, $P^{(n)}$ is canonically isomorphic to the $k$-subspace $P^{N_n,\psi_k}$ consisting of elements on which $N_n$ acts via $\psi_k$.
Thus we get the following string of $k$-isomorphisms:
\begin{align*}
P^{(n)} &\cong \Hom_{k[N_n]}(\psi_k,P)\\
&\cong \Hom_{k[G_n]}(\Ind_{N_n}^{G_n}\psi_k,P)\\
&\cong \End_{k[G_n]}(P) = R(\pi)
\end{align*}
The first isomorphism is the defining property of $(N_n,\psi_k)$-invariants, and takes an element $v\in P^{N_n,\psi_k}$ to the map $\psi_k \to P$ defined by sending $1$ to $v$. The second isomorphism is Frobenius reciprocity (\autoref{section_representations}) which takes a homomorphism $\phi:\psi_k\to P$ to a homomorphism $\Phi:\Ind_{N_n}^{G_n} \psi_k\to P$ defined by $\Phi(f) = \phi(f(1))=f(1)\phi(1)$ for $f\in \Ind_{N_n}^{G_n} \psi_k$. The third isomorphism follows from multiplicity-freeness of $\Ind_{N_n}^{G_n} \psi_k$, which implies any homomorphism $\Ind_{N_n}^{G_n} \psi_k\to P$ is zero on the summands distinct from $P$.

The ring $R=R(\pi)$ is, by definition, the $k[G_n]$-linear endomorphism ring of $P$, so acts naturally on $P^{N_n,\psi_k}$. In addition, $R$ acts on each of the above $\Hom$ spaces by acting on the target. We check that the composite of the first two isomorphisms above is $R$-linear; the others are clear. Given $r\in R$, let $v\in P^{N_n,\psi_k}$, and take $\Phi\in \Hom_{k[G_n]}(\Ind_{N_n}^{G_n} \psi_k,P)$ such that $v$ is sent to $\Phi$ in the setting of the previous paragraph. We have that $r(v)$ maps to the homomorphism 
$$f\mapsto f(1)r(v) = r(f(1)v) = r(\Phi(f)),$$ because $r$ is $k$-linear.
\end{proof}

\subsection{Whittaker model of \texorpdfstring{$P(\pi)$}{P(pi)}}\label{subsection_whittaker_model}

We first note that $P=P(\pi)$ is of $\psi$-Whittaker type: from \autoref{prop:Pderivative}, $P^{(n)}$ is an $R(\pi)$-module of rank $1$, so that \[ \Hom_{R(\pi)}(P^{(n)}, R(\pi)) =  R(\pi).\]

This allows us to consider the Whittaker model of $P(\pi)$ in $\Ind_{N_n}^{G_n}\psi_{R(\pi)}$, which entails choosing an element $\eta\in \Hom_{R(\pi)}(P^{(n)}, R(\pi))$ corresponding to a unit in $R(\pi)$. If we identify $P(\pi)^{(n)}\cong R(\pi)$ under the isomorphism from \autoref{prop:Pderivative}, and thus identify
    \[ \Hom_{R(\pi)}(P^{(n)}, R(\pi))\cong \Hom_{R(\pi)}(R(\pi),R(\pi)) = R(\pi), \]
we might as well choose $\eta$ corresponding to the identity under this identification. The Whittaker model $\mathcal{W}(P(\pi),\psi_{R(\pi)})$ is then, by definition, the image of the map
\begin{align*}
P(\pi)&\to \Ind_{N_n}^{G_n}\psi_{R(\pi)}\\
f&\mapsto W_f\ 
\end{align*}
defined as follows. If $\lambda: P(\pi) \to P(\pi)^{(n)}\cong R(\pi)$ denotes the natural quotient map, Frobenius reciprocity gives the formula $$W_f(g):= \lambda(g f),\ \ g\in G_n.$$

Next, we will compute a natural section of the above map from $P(\pi)$ to its Whittaker model. There is a canonical map of $k[N_n]$-modules
$$\Ind_{N_n}^{G_n}\psi_k\to \psi_k$$ given by evaluation at the identity.
For each irreducible generic representation $\pi$, we can restrict this to a map $P(\pi)\to \psi_k$, which must factor through the $(N_n,\psi_k)$-coinvariants to give a map
    \[ \theta:R(\pi)\cong P(\pi)^{(n)} \to k \]
of $k$-vector spaces. In other words, for $f$ an element of $\Ind_{N_n}^{G_n}\psi_k$ that lives in $P(\pi)$, we have $\theta(\lambda(f)) = f(1)$. 

Let $W_f \in \mathcal{W}(P(\pi),\psi_{R(\pi)})$ be the $R(\pi)$-valued Whittaker function of $f$. We have
$$(\theta\circ W_f)(g) = \theta(\lambda(g f)) = (g f)(1) = f(g),\ \ \ g\in G_n.$$ Thus a section of our chosen map
$$P(\pi)\to \Ind_{N_n}^{G_n}\psi_{R(\pi)}$$ is given by composing with $\theta$. We record these observations in the following corollary.

\begin{corollary}
   The representation $P = P(\pi)$ is of $\psi$-Whittaker type and essentially $\psi$-generic, i.e., embeds in its Whittaker model.
\end{corollary}

\subsubsection{Alternative view of the Whittaker model of \texorpdfstring{$P(\pi)$}{P(pi)}}
\label{subsubsec:alternative view}

In this subsection we attempt to illustrate why $\mathcal{W}(P(\pi),\psi_{R(\pi)})$ is more useful that $P(\pi)$ itself: it sees the natural action of $R(\pi)$ on $P(\pi)$ in both its $G_n$-structure and its $R(\pi)$-structure coming from multiplying $R(\pi)$-valued functions by elements of $R(\pi)$. We will not use these results in the rest of the paper, but include them to give a more conceptual understanding of the map $\theta$ from the previous subsection.

By extending scalars along the natural inclusion $k\subset R(\pi)$ we get an embedding $\Ind_{N_n}^{G_n}\psi_k \hookrightarrow \Ind_{N_n}^{G_n}\psi_{R(\pi)}$, which restricts to an inclusion on the summand $P = P(\pi)$
$$P\hookrightarrow P\otimes_kR(\pi).$$ The module $P\otimes_k R(\pi)$ has two distinct $R(\pi)$-module structures, both of which commute with the $G_n$-action, namely the one defined on simple tensors by
$$\phi*(f\otimes \phi') = \phi(f)\otimes \phi',$$ and the one defined by
$$\phi\cdot(f\otimes\phi') = f\otimes (\phi\cdot \phi').$$ There is a natural projection of $k[G_n]$-modules
$$\varpi: P\otimes_k R(\pi)\to P\otimes_{R(\pi)}R(\pi)\cong P$$ given by taking the quotient by the $k$-subspace $\ker(\varpi)$ generated by tensors of the form $f\otimes \phi -\phi(f)\otimes 1$. 

Since $P$ is a projective $k[G_n]$-module, $\varpi$ is a split surjection, and there exists a section $\eta:P\to P\otimes_k R(\pi)$ giving a decomposition into a direct sum of $k[G_n]$-modules
$$P\otimes_kR(\pi) = \eta(P)\oplus \ker(\varpi).$$ However, by commutativity of $R(\pi)$ we have
\begin{align*}
\phi*(f\otimes \phi' - \phi'(f)\otimes 1) &= \phi(f)\otimes \phi' - \phi'(\phi(f))\otimes 1\\
\phi\cdot (f\otimes \phi'-\phi'(f)\otimes 1) & = (f\otimes \phi\phi' - \phi\phi'(f)\otimes 1) - (\phi'(f)\otimes \phi - \phi(\phi'(f))\otimes 1),
\end{align*}
so $\ker(\varpi)$ is stable under $R(\pi)$ for both actions, and thus so is $\eta(P)$. We conclude the above splitting is in fact a splitting of $R(\pi)[G_n]$-modules for both $R(\pi)$ actions. Furthermore, given $f\in P$, we must have that $\phi*\eta(f) - \phi\cdot\eta(f)$ is an element of $\ker(\varpi)\cap \eta(P) = \{0\}$, which shows $\phi*\eta(f) = \phi\cdot \eta(f)$.

We conclude that each splitting $\eta$ gives rise to a Whittaker model
$$P \hookrightarrow \Ind_{N_n}^{G_n}\psi_{R(\pi)},$$ for which $\varpi$ is a canonical section, and whose image lands in the subset
$$\{W\in P\otimes_kR(\pi)\ :\ \phi* W = \phi\cdot W\ ,\ \phi\in R(\pi)\}.$$ From this perspective, the relation
$$\theta\circ W_f = f$$ of the previous subsection amounts to the fact that the composite map
$$P \hookrightarrow P\otimes_k R(\pi) \xrightarrow{\varpi} P$$ is equivalent to the identity.

\subsection{Definition of the new gamma factors}

Given irreducible generic $k$-representations~$\rho$ of $G_n$ and $\pi$ of $G_m$, we will define a modified gamma factor $\tilde{\gamma}(\rho\times\pi,\psi)$ as follows. 

Let $\rho_{R(\pi)}$ denote the extension of scalars $\rho\otimes_kR(\pi)$ along the structure morphism $k\hookrightarrow R(\pi)$. Now since $(\rho_{R(\pi)})^{(n)} = \rho^{(n)}\otimes_k R(\pi) \cong k\otimes_k R(\pi)\cong R(\pi)$, and $P(\pi)^{(m)}$ is free of rank one over $R(\pi)$ by \autoref{prop:Pderivative}, we may apply \autoref{cor:functional_equation} to the $R(\pi)[G_n]$-module~$\rho_{R(\pi)}$ and the $R(\pi)[G_m]$-module $P(\pi)$:
$$\tilde{\gamma}(\rho\times\pi,\psi) := \gamma(\rho_{R(\pi)}\times P(\pi),\psi)\in R(\pi)^{\times}.$$

\subsection{Completeness of Whittaker models} \label{completeness}
To prove the converse theorem we need a so-called ``$L^2$-completeness of Whittaker models'' statement. The point of passing to $R(\pi)$ coefficients instead of $k$ coefficients is to recover such a completeness statement.  \autoref{sec:sage} discusses counterexamples to the converse theorem for $k$-valued gamma factors: they arise because of the failure of completeness of Whittaker models (which for $G_1$ reduces to the dual to linear independence of characters.)

\begin{theorem}
\label{thm:completeness}
Let $H$ be an element of $\Ind_{N_n}^{G_n}\psi_k$. If
$$\sum_{x\in {N_n}\backslash G_n}H(x)W(x)=0$$ for every $W\in \mathcal{W}(P(\pi),\psi^{-1}_{R(\pi)})$, for every irreducible generic representation $\pi$ of $G_n$, then~$H$ is identically zero.
\end{theorem}
\begin{proof}
By replacing $\psi$ with $\psi^{-1}$ in \autoref{subsection_whittaker_model} we can make a choice of isomorphism $P(\pi)^{(n,\psi^{-1})}\cong R(\pi)$ for each irreducible $\psi^{-1}$-generic representation $\pi$ to get a Whittaker model 
\begin{align*}
P(\pi)&\to \Ind_{N_n}^{G_n}\psi^{-1}_{R(\pi)}\\
f&\mapsto W_f.
\end{align*}
Recall from \autoref{subsection_whittaker_model} there is a map $\theta:R(\pi) \cong P_{N_n,\psi^{-1}} \to k$ arising from $f\mapsto f(1)$ such that $f = \theta\circ W_f$.

Thus for every such $f\in R(\pi)$, we have
\begin{align*}
0& = \theta\left(\sum_{x\in N_n\backslash G_n}H(x)W_f(x)\right)\\
&= \sum_{x\in N_n\backslash G_n}H(x)\theta(W_f(x))\\
&=\sum_{x\in N_n\backslash G_n}H(x)f(x).
\end{align*}
We established in \autoref{subsec: projective envelopes} that $\Ind_{N_n}^{G_n} \psi_k$ is the multiplicity-free direct sum of the $P(\pi)$ for $\pi$ irreducible generic, and as such is spanned by $f \in P(\pi)$.  Since 
\begin{align*}
\Ind_{N_n}^{G_n}\psi_k\times \Ind_{N_n}^{G_n}\psi_k^{-1}&\to k\\
(H,f)&\to \sum_{x\in N_n\backslash G_n}H(x)f(x)
\end{align*}
is a nondegenerate duality pairing (\cite[\S I.5.11]{vig_book}), we conclude that $H$ is identically zero. 
\end{proof}

\subsection{Proof of converse theorem}\label{sec:proof_converse}
We finally arrive at the proof of \autoref{mainthm_converse}. Our strategy is inspired by the proof of the converse theorem in \cite{henniart}.

If $\rho_1$ and $\rho_2$ are irreducible cuspidal $k$-representations of $G_n$, set $$S(\rho_1,\rho_2,\psi) := \left\{(W_1,W_2)\in \mathcal{W}(\rho_1,\psi_k)\times\mathcal{W}(\rho_2,\psi_k): \restr{W_1}{P_n} = \restr{W_2}{P_n}\right\}.$$
There is a diagonal action of $P_n$ on $\mathcal{W}(\rho_1,\psi_k)\times\mathcal{W}(\rho_2,\psi_k)$ and the subspace $S(\rho_1,\rho_2,\psi)$ is stable under this action by its definition. We will show it is in fact $G_n$-stable if we suppose that $\rho_1$ and $\rho_2$ have the same gamma factors.
\begin{proposition}
\label{prop:stabilityS}
Let $\rho_1$ and $\rho_2$ be irreducible cuspidal $k$-representations of $G_n$ and suppose that
$$\tilde{\gamma}(\rho_1\times \pi,\psi) = \tilde{\gamma}(\rho_2\times\pi,\psi)$$ for all irreducible generic representations $\pi$ of $G_{n-1}$. Then $S(\rho_1,\rho_2,\psi)$ is stable under the diagonal action of $G_n$.
\end{proposition}
\begin{proof}
The restriction of a Whittaker function to $P_n$ is determined by its values on $G_{n-1}$ (embedded in $G_n$ in the top left). Therefore:
\begin{align*}
    & (W_1,W_2)\in S(\rho_1,\rho_2,\psi) \\
    &\Leftrightarrow W_1\left(\begin{smallmatrix}x&0\\0&1 \end{smallmatrix}\right) = W_2\left(\begin{smallmatrix}x&0\\0&1 \end{smallmatrix}\right) \text{ for all } x \in G_{n-1}\\
    &\xLeftrightarrow{\text{\autoref{thm:completeness}}} \sum_{x\in N_{n-1}\backslash G_{n-1}}W_1\left(\begin{smallmatrix}x&0\\0&1 \end{smallmatrix}\right)W'(x) = \sum_{x\in N_{n-1}\backslash G_{n-1}}W_2\left(\begin{smallmatrix}x&0\\0&1 \end{smallmatrix}\right)W'(x)\\
    & \text{for all $W'\in\mathcal{W}(P(\pi),\psi_{R(\pi)}^{-1})$, for all $\pi$}\\
    &\xLeftrightarrow{\text{equality of $\tilde{\gamma}$'s}} \sum_{x\in N_{n-1}\backslash G_{n-1}}W_1\left(\begin{smallmatrix}0&1\\x&0 \end{smallmatrix}\right)W'(x) = \sum_{x\in N_{n-1}\backslash G_{n-1}}W_2\left(\begin{smallmatrix}0&1\\x&0 \end{smallmatrix}\right)W'(x)\\
    & \text{for all $W'\in\mathcal{W}(P(\pi),\psi_{R(\pi)}^{-1})$, for all $\pi$}\\
    &\xLeftrightarrow{\text{\autoref{thm:completeness}}} W_1\left(\begin{smallmatrix}0&1\\{^\iota x}&0 \end{smallmatrix}\right) = W_2\left(\begin{smallmatrix}0&1\\{^\iota x}&0 \end{smallmatrix}\right) \text{ for all } x \in G_{n-1}\\
    &\Leftrightarrow \widetilde{W_1}\left(\begin{smallmatrix}x&0\\0&1 \end{smallmatrix}\right) = \widetilde{W_2}\left(\begin{smallmatrix}x&0\\0&1 \end{smallmatrix}\right) \text{ for all } x \in G_{n-1}\\
    &\Leftrightarrow (\widetilde{W_1},\widetilde{W_2})\in S(\rho_1^{\vee},\rho_2^{\vee},\psi^{-1})
\end{align*}
(Note that we are applying 
\autoref{cor:functional_equation} in the step labeled ``equality of $\widetilde{\gamma}$'s'' above). Now if $\overline{p}\in {^tP_n}$ we have, for $i=1,2$,
    \begin{align} \label{whit model transpose}
    \begin{split}
       \wt{\overline{p}W_i}(g) &= (\overline{p}W_i)(\begin{psmallmatrix}0&1\\1&0\end{psmallmatrix}(^\iota g)) \\
        &= W_i(\begin{psmallmatrix}0&1\\1&0\end{psmallmatrix}(^\iota g)\overline{p}) \\
        &= W_i(\begin{psmallmatrix}0&1\\1&0\end{psmallmatrix}(^\iota(g(^\iota\overline{p})))) \\
        &= \wt{W}_i(g(^\iota\overline{p})) = (^\iota\overline{p}\wt{W}_i)(g).
    \end{split}
    \end{align}

Thus if $(W_1,W_2)\in S(\rho_1,\rho_2,\psi)$, then since $S(\rho_1^{\vee},\rho_2^{\vee},\psi^{-1})$ is $P_n$-stable we have
$$(\widetilde{\overline{p}W_1},\widetilde{\overline{p}W_2}) = ({^\iota\overline{p}}\widetilde{W_1},{^\iota\overline{p}}\widetilde{W_2})\in S(\rho_1^{\vee},\rho_2^{\vee},\psi^{-1}).$$ The above equivalences then imply that $(\overline{p}W_1,\overline{p}W_2)$ is in $S(\rho_1,\rho_2,\psi)$. Thus we have shown that $S(\rho_1,\rho_2,\psi)$ is stable under both $P_n$ and $^tP_n$. Since these two groups generate $G_n$ we conclude that $S(\rho_1,\rho_2,\psi)$ is stable under $G_n$.
\end{proof}

\begin{corollary}\label{corollary restriction to Pn}
Suppose $\rho_1$ and $\rho_2$ are irreducible cuspidal representations of $G_n$ over $k$ and suppose that 
    \[ \tilde{\gamma}(\rho_1\times\pi,\psi) = \tilde{\gamma}(\rho_2\times\pi,\psi) \]
for every irreducible generic representation $\pi$ of $G_{n-1}$. Let $W_1$, $W_2$ be elements of the Whittaker spaces $\calw(\rho_1,\psi_k)$, $\calw(\rho_2,\psi_k)$, respectively. Then the following equivalence holds
    \[ W_1|_{P_n} = W_2|_{P_n}\text{ if and only if }W_1 = W_2. \]
\end{corollary}\label{kirillov_model}
\begin{proof}
Let $W_1 \in \calw(\rho_1,\psi_k)$ and $W_2 \in \calw(\rho_2,\psi_k)$
such that $\restr{W_1}{P_n} = \restr{W_2}{P_n}$.
Then for all $g \in G_n$, \autoref{prop:stabilityS} implies that $\restr{(gW_1)}{P_n} = \restr{(gW_2)}{P_n}$. Evaluating at the identity, we see that
    \[ W_1(g) = (gW_1)(1) = (gW_2)(1) = W_2(g), \]
so $W_1=W_2$.
\end{proof}
\begin{corollary}
Let $\rho \in \Rep_k(G_n)$ be irreducible cuspidal and fix $W \in \mathcal{W}(\rho,\psi_k)$. If $W|_{P_n}=0$ then $W=0$. In other words the map 
\begin{align*}
    \mathcal{W}(\rho,\psi_k)&\to \Ind_{N_n}^{P_n}\psi_k\\
    W&\mapsto W|_{P_n}
\end{align*}
is injective (hence an isomorphism of $P_n$-modules by \autoref{lem:restriction_to_pn}).
\end{corollary}

\begin{corollary}
Let $\rho_1$ and $\rho_2$ be irreducible cuspidal $k$-representations of $G_n$ and suppose that
$$\tilde{\gamma}(\rho_1\times \pi,\psi) = \tilde{\gamma}(\rho_2\times\pi,\psi)$$ for all irreducible generic representations $\pi$ of $G_{n-1}$. Then $\rho_1\cong \rho_2$.
\end{corollary}
\begin{proof}
By the previous corollary, for every $W_1\in \mathcal{W}(\rho_1,\psi_k)$ there is a unique $W_2\in \mathcal{W}(\rho_2,\psi_k)$ such that $W_1|_{P_n}=W_2|_{P_n}$. This gives a morphism of $k[G_n]$-modules $\mathcal{W}(\rho_1,\psi_k)\to S(\rho_1,\rho_2,\psi)$. Projection on the second factor gives a composite morphism    \[ \mathcal{W}(\rho_1,\psi_k) \to S(\rho_1,\rho_2,\psi_k)\to \mathcal{W}(\rho_2,\psi_k), \]
which is nonzero and $G_n$-equivariant. Since $\rho_1$ and $\rho_2$ are irreducible it follows that $\rho_1\cong \rho_2$. 
\end{proof}

\begin{appendices}


\addtocontents{toc}{\protect\setcounter{tocdepth}{1}}
\section{Counterexamples to the na\"ive converse theorem}
\label{sec:sage}

We used Sage to discover counterexamples to the na\"ive converse theorem mod~$\ell$ for~$\GL_2(\F_q)$, following the explicit computations for gamma factors in Theorem 21.1 of \cite{ps}. The code can be found in \cite{Counterexample} Our main function \texttt{gammafactors} computes $\overline{\F}_\ell$-valued gamma factors. In the $GL_2(\F_q)$ context, the ``na\"ive converse theorem" refers to the following:

\begin{ncs*}[Na\"ive Converse Theorem]
Let $k$ be an algebraically closed field of characteristic $\ell$. Let $\rho$ an irreducible generic representation of $GL_2(\F_q)$, $\omega$ a character of $GL_1(\F_q)$, and $\psi$ a nontrival character of~$\F_q$, all $k$-valued. For the gamma factor $\gamma(\rho \times \omega,\psi):=\gamma(\rho,\omega,\psi)$ defined by \autoref{mainthm_functionaleqn}, if
\[ \gamma(\rho \times \omega,\psi) = \gamma(\rho' \times \omega,\psi) \text{ for all }\omega \in \widehat{GL_1(\F_q)},\]  then $\rho \cong \rho'$. 
\end{ncs*}

We found counterexamples to the na\"ive converse theorem with $k = \bar{\F}_\ell$ for the pairs  \[ (\ell, q) = (2,5), (2,17), (3,7), (3,19), (5,11), (11,23), (23,47), (29,59). \] In all of these situations, $q = 2\ell^i+1$ for some positive integer $i$. Informed by this data, we make the following: 
\begin{conjecture*} The na\"ive converse theorem for mod $\ell$ representations of $\GL_2(\mathbb{F}_q)$ fails exactly when $q = 2\ell^i+1$ for some value of $i>0$. 
\end{conjecture*}

Below, we make this conjecture precise by defining the na\"ive mod $\ell$ gamma factor. We then describe the algorithm through which we found the counterexamples. 

\subsection{Mod \texorpdfstring{$\ell$}{ell} gamma factors of \texorpdfstring{$\GL_2(\mathbb{F}_q)$}{GL(2,Fq)}}

Our computations rely on the explicit realizations of gamma factors of irreducible cuspidal representations of  $\GL_2(\mathbb{F}_q)$ as Gauss sums.  

The specialization of \autoref{section functional equation} and \autoref{mainthm_functionaleqn} to $n=2$, and $m=1$ 
recovers the construction of \cite{ps}, and extends them to representations valued in any Noetherian~$\Z[1/p,\zeta_p]$-algebra $R$. For simplicity, assume $R$ is a field, let $\rho$ be an irreducible generic $R$-representation of $\GL_2(\F_q)$, and let $\omega$ be a character of $\GL_1(\F_q) = \F_q^\times$ not exceptional for $\rho$. Then $\gamma(\rho \times \omega, \psi)$ is defined by the functional equation
\begin{equation}
\label{eq def gamma naive} 
    \gamma(\rho \times \omega, \psi) \sum_{x \in \F_q^{\times}} W\begin{psmallmatrix}x&0\\0&1\end{psmallmatrix} \, \omega(x) = \sum_{x \in \F_q^{\times}} W\begin{psmallmatrix}0&1\\x&0\end{psmallmatrix} \, \omega(x),
\end{equation}
for any $W \in \mathcal{W}(\rho,\psi)$.

Let $R = \overline{\F}_\ell$ with~$(q,\ell) = 1$ and $(\rho,V)$ be irreducible cuspidal. Vigneras \cite{vigneras} constructs~$\rho=\rho_\nu$ from a character $\nu$ of $\F_{q^2}^\times$. There is an identification
    \[ V|_{P_2} \;\cong \;\Ind_{N_2}^{P_2}\psi \;\cong\; \{f: \F_q^\times \to \overline{\F}_\ell\}, \]
where the first isomorphism follows from \autoref{CuspidalCriterionVigneras} and the second is restriction to $\F_q^\times \leq P_2$. In these coordinates, there is a unique Bessel vector $f \in V$ satisfying
    \[ f(x) = \delta_{x=1}, \quad W_f\begin{psmallmatrix}x&0\\0&1\end{psmallmatrix}=\delta_{x=1}, \quad \text{and} \quad \rho(n)f=\psi(n)f, \; n \in N_2. \]
The second property together with the functional equation imply that
    \[ \gamma(\rho_\nu \times \omega, \psi) = \sum_{x \in \F_q^{\times}} W_f\begin{psmallmatrix}0&1\\x&0\end{psmallmatrix}\omega(x) .\]
Using the properties of $f$, we replicate the computations of \cite[\S 21]{ps} using the constructions of \cite{vigneras} to recover
\[ \gamma(\rho_\nu \times \omega, \psi) = q^{-1}\nu(-1)\sum_{t \in \F_{q^2}^\times} \nu(t)\omega(t\overline{t})^{-1}\psi(t+\overline{t}), \] for $\overline{t} = t^q$. This realizes the na\"ive mod $\ell$ gamma factor as a Gauss sum.

\subsection{The algorithm}
The algorithm executes two tasks: 
\begin{enumerate}
    \item The function \texttt{gammafactor} computes gamma factors. 
    \item The function \texttt{iscounterexample} detects equalities between gamma factors. \\ 
\end{enumerate}

\paragraph{\bf The function \texttt{gammafactor}.} Let $q$ be prime. To compute the Gauss sums, we exploit that all groups in sight are cyclic. We have the following variables: 
\begin{itemize}
    \item \texttt{w} is a choice of generator of $\F_{q^2}^\times$,
    \item \texttt{coprime$\_$m} (resp. \texttt{coprime$\_$n}) is the largest divisor of $q^2-1$ (resp. $q-1$) coprime to $\ell$.
    \item \texttt{nu} is a choice of primitive root of unity of order \texttt{coprime$\_$m}
    \item \texttt{omega} is a choice of primitive root of unity of order \texttt{coprime$\_$n}
    \item \texttt{psi} is a choice of primitive $q^{\text{th}}$ root of unity, and we fix the additive character $\psi: \F_q \to \overline{\F}_\ell^\times$ so that $\psi(1) = \texttt{psi}$.
\end{itemize}
This allows us to identify characters of $\F_{q^2}^\times $ and $\F_{q}^\times$ with integers in the relevant ranges as follows: \begin{itemize}
    \item For $i \in [0,\texttt{coprime}\_\texttt{m}-1]$, the character $\nu_i$ of $\F_{q^2}^\times$ is defined by \[ \nu_i(\texttt{w}) = \texttt{nu}^i. \] We will denote the cuspidal representation $\rho_{\nu_i}$ by $\rho_i$.
    \item For $j \in [0,\texttt{coprime}\_\texttt{n}-1]$, the character $\omega_j$ of $\F_{q}^\times$ is defined by \[ \nu_i(\texttt{w}^{q+1}) = \texttt{omega}^j. \]  Note that $\texttt{w}^{q+1}$ is a generator of $\F_q^\times$.
    \end{itemize}
The input of the function \texttt{gammafactor} is the triple $(\texttt{q},\texttt{i},\texttt{j})$.  Letting $\texttt{f}= \texttt{nu} \, \hat{} \, (\texttt{Integer}(\texttt{m}/\texttt{2}))$, the function returns
\[
\texttt{gf} := \texttt{f} * 
  \texttt{sum}(\texttt{nu}\,\hat{}\,(\texttt{i}*\texttt{k})*\texttt{omega}\,\hat{}\,(\texttt{j}*\texttt{k})*\texttt{psi}\,\hat{}\,(\texttt{w}\,\hat{\,}\texttt{k}+\texttt{w}\,\hat{}\,(\texttt{q}*\texttt{k})) \texttt{ for}\; \texttt{k}\; \texttt{in } [\texttt{0}..\texttt{m-1}]) 
\]
which computes
\[ q \cdot \gamma(\rho_i,\omega_j) = \nu_i(-1)\sum_{k=0}^{m-1}\nu_i(\texttt{w}^k)\omega_j(\texttt{w}^{(q+1)\cdot k})\psi(\texttt{w}^k+\texttt{w}^{qk}).\]\\

\paragraph{\bf The function \texttt{iscounterexample}.}

This function compares the output of the function \texttt{gammafactor(q,i,j)} for different values of \texttt{i} and \texttt{j}. First recall that, $\rho_i \cong \rho_{i'}$ for $i \neq i'$ precisely when $\nu_i = \overline{\nu}_{i'}$, i.e. if $i' \equiv q\cdot i$ mod $m$.

The function \texttt{gammafactorarray} first runs over all isomorphism classes of irreducible cuspidal representations $\rho_j$, removes duplicates, and records a list \texttt{nonConjChars} of integers~\texttt{j} corresponding to a list of non-duplicate $\rho_j$. In order to reduce runtime and avoid computing unnecessary gamma factors, it next computes \texttt{gammafactor(q,i,0)} for all values of \texttt{i} in the list \texttt{nonConjChars}. If two values \texttt{i} and \texttt{i'} have the same gamma factor $\gamma(\rho_i,\one)$, they are added to the list \texttt{potential$\_$duplicates}. Finally, the function returns an array \texttt{gammafactorarray} of \texttt{gammafactor(q,i,j)} for all \texttt{i} in \texttt{potential$\_$duplicates} and \texttt{j} in the range \texttt{[0,\texttt{coprime}$\_$\texttt{n}$-1$]}. 

The function \texttt{iscounterexample} then runs the utility function \texttt{findduplicates}, which takes as an input an array and returns a list of duplicates among the rows of the array. Finally, the function \texttt{iscounterexample} prints the list of duplicates. 

Currently, the speed of the algorithm is restricted by the actual computations of the Gauss sums, which runs at least in $O(q^2)$.

\section{Example of \texorpdfstring{$\widetilde{\gamma}(\pi\times \pi',\psi)$}{new gamma factor} for \texorpdfstring{$\GL_2$}{GL(2)}} \label{sec:computing new gamma}

In this appendix we give an example illustrating 
\autoref{mainthm_converse} in the case $n=2$, $m=1$, $q=5$, and $\ell=2$. More precisely, we compute the new gamma factors  $\widetilde{\gamma}(\rho_1\times \omega,\psi)$ and $\widetilde{\gamma}(\rho_\nu\times \omega,\psi)$ of \autoref{def:new-gamma-factor} when $\rho_1$ and $\rho_\nu$ are the irreducible cuspidal representations of $\GL_2(\F_q)$ occurring in the first counterexample in the list in \autoref{sec:sage}.

Suppose $n=2$, $m=1$, $q=5$, and $\ell=2$. Since $\flbar^\times$ has no elements of order $2$ or $4$, there is only one irreducible representation $\omega$ of $G_1=\GL_1(\F_q)$, namely the trivial character $\one$. Since the unipotent subgroup $N_1$ is trivial, $\one$ is generic and the Whittaker space $\Ind_{N_1}^{G_1}(\psi)$ is simply the space of functions $\{f:\F_q^{\times}\to \flbar \}$ with the action of right-translation. There is an isomorphism $\lambda$ of $\flbar[\F_q^{\times}]$-modules, 
\begin{align*}
\lambda: \Ind_{N_1}^{G_1}(\psi)&\to \flbar[\F_q^{\times}]\\
f&\mapsto \sum_{x\in \F_q^{\times}}f(x)x\ ,
\end{align*}
for the action of $\F_q^{\times}$ on the target given by 
    \[ g(\sum_{x\in \F_q^{\times}}\alpha_x x) = \sum_{x\in \F_q^{\times}}\alpha_x g^{-1}x. \]
Thus $\Ind_{N_1}^{G_1}(\psi)$ is naturally isomorphic to the free (hence projective) module $\flbar[\F_q^{\times}]$. By \autoref{subsec: projective envelopes}, the indecomposable summands of $\Ind_{N_1}^{G_1}(\psi)$ are precisely the projective envelopes of the distinct irreducible $\psi$-generic representations of $G_1$, of which there is only one, namely $\one$. Thus $\Ind_{N_1}^{G_1}(\psi)$ is indecomposable and is the projective envelope $P(\one)$. The endomorphism ring $R(\one)$ is $$\End_{\flbar[\F_q^{\times}]}(\flbar[\F_q^{\times}]) = \flbar[\F_q^{\times}].$$ The gamma factors $\widetilde{\gamma}(\rho\times \one)$ are elements of the commutative ring $R(\one)=\flbar[\F_q^{\times}]$; our aim is to compute them. 

We must consider $P(\one)$ as a representation of $G_1$ with coefficients in the ring $R(\one)$ and compute its $R(\one)$-valued Whittaker model. The Bernstein-Zelevinsky derivative $P(\one)^{(1)}$ is precisely $P(\one)_{N_1,\psi} = P(\one)$ so the map $\lambda$ above (considered only as a map of $\flbar$-vector spaces) defines by Frobenius reciprocity a map
\begin{align*}
P(\one)&\to \Ind_{N_1}^{G_1}(\psi_{R(\one)})\\
f &\mapsto W_f\ ,
\end{align*}
where $W_f(x) = \lambda(xf)$. In particular, $W_f(1) = \lambda(f)\in R(\one)$. On the other hand, \autoref{subsubsec:alternative view} gives a way to describe $W_f$ completely. Namely, given $g\in \F_q^{\times}$, we must have $(g*W_f)(1)=W_f(g)$ and also $(g*W_f)(1) = g\cdot W_f(1)=g\lambda(f)$ (multiplication in the group ring $R(\one)$), so we have
$$W_f(g) = g\lambda(f).$$ (Note that composing $W_f$ with the evaluation-at-identity map $\theta:P(\one)\to \flbar$ returns $f$, c.f. \autoref{subsection_whittaker_model}). 

Shortly, we will make use of a particular $R(\one)$-valued Whittaker function of $P(\one)$. If $f$ is given by $\delta_{x=1}$ we set $W_1 = W_{\delta_{x=1}}$. Letting $g$ be a generator of $\F_q^{\times}$, we have, for $i=0,1,2,3$,
$$W_1(g^i) = g^i\lambda(f) = g^i(1 + 0g + 0g^2 + 0g^3) = g^i.$$

Given an irreducible cuspidal representation $(\rho,V)$ of $\GL_2(\F_q)$ we can compute $\widetilde{\gamma}(\rho\times \one)$ in a similar manner to \autoref{sec:sage}, while working over $\flbar[\F_q^{\times}]$ instead of $\flbar$. As in \cite[\S 19]{ps} the Bessel function $J_\rho\in \calw(\rho,\psi)$ is defined by the property $J_\rho\left(\begin{smallmatrix} x &\\&1\end{smallmatrix}\right) = \delta_{x=1}$ and that it scales by $\psi(n)$ under both left and right-translation by $n\in N_2$. For such $J_{\rho}$ we have
\begin{align*}
\widetilde{\gamma}(\rho_1\times\one) &= \sum_{x\in \F_q^{\times}}J_\rho\left(\begin{smallmatrix} 0&1\\x&1\end{smallmatrix}\right)W_1(x)
\end{align*}
where $W_1(x)$ is the element of $\calw(P(\one), \psi_{R(\one)}^{-1})$ described above.

Let $\rho_1$ be the cuspidal representation of $\GL_2(\F_q)$ coming from the trivial character $\F_{q^2}^{\times}\to \flbar^{\times}$; it is the irreducible cuspidal representation occuring as a subquotient of the length-three module $$\Ind_{B_2}^{G_2}(\one)= \{h:\GL_2(\F_q)\to \flbar\ :\ h(bg) = h(g),\ b\in B_2,\ g\in G_2\},$$  which has the trivial character as a sub (the constant functions) and as a quotient (\cite{vigneras}). The method of \cite[\S 13]{ps} works in this setting to construct cuspidals (\cite[Thm 2(a)]{vigneras}), so following the computation in \cite[\S 21]{ps}, we find
\begin{align*}
\widetilde{\gamma}(\rho_1\times \one) &= \gamma(\rho_1\times P(\one)) \\
&= q^{-1}\sum_{t\in \F_{q^2}^{\times}}\psi(t+t^q)W_1((t^{1+q})^{-1})\\
&= q^{-1}\sum_{t\in \F_{q^2}^{\times}}\psi(t+t^q)(t^{1+q})^{-1},
\end{align*}
viewed as an element of the group ring $\flbar[\F_q^{\times}]$. On the other hand, let $\nu:\F_{q^2}^{\times}\to \flbar^{\times}$ be the character sending a generator to a primitive cube root of unity and let $\rho_{\nu}$ be the corresponding cuspidal representation of $\GL_2(\F_q)$. Again following the computation in \cite[\S 21]{ps}, we find 
\begin{align*}
\widetilde{\gamma}(\rho_{\nu}\times \one) 
&= q^{-1}\nu(-1)\sum_{t\in \F_{q^2}^{\times}}\nu(t)\psi(t+t^q)W_1((t^{1+q})^{-1})\\
&= q^{-1}\nu(-1)\sum_{t\in \F_{q^2}^{\times}}\nu(t)\psi(t+t^q)(t^{1+q})^{-1}\ .
\end{align*}
In this context, \autoref{mainthm_converse} guarantees that, since $\rho_{\nu}$ and $\rho_1$ are distinct irreducible cuspidals, $\widetilde{\gamma}(\rho_1\times \one)\neq \widetilde{\gamma}(\rho_{\nu}\times \one)$ in the ring $\flbar[\F_q^{\times}]$. This can be verified by direct computation. By contrast, the na\"ive gamma factors can be calculated,
\begin{align*}
\gamma(\rho_1\times \one) &= q^{-1}\sum_{t\in \F_{q^2}^{\times}}\psi(t+t^q)\one((t^{1+q})^{-1})\\
&= q^{-1}\sum_{t\in \F_{q^2}^{\times}}\psi(t+t^q)\\
\gamma(\rho_\nu\times \one) &= q^{-1}\nu(-1)\sum_{t\in \F_{q^2}^{\times}}\nu(t)\psi(t+t^q)\one((t^{1+q})^{-1})\\
&=q^{-1}\nu(-1)\sum_{t\in \F_{q^2}^{\times}}\nu(t)\psi(t+t^q)\ ,
\end{align*}
and both $\gamma(\rho_1\times \one)$ and $\gamma(\rho_{\nu}\times\one)$ are equivalent to $1$ mod $\ell$.

\section{\texorpdfstring{$\ell$}{ell}-regular gamma factors for \texorpdfstring{$\GL_2$}{GL(2)}} \label{sec: ell-regular gamma factors}

In this appendix we construct an ``$\ell$-regular'' gamma factor for pairs $(\rho,\omega)$ where $\rho$ is a mod~$\ell$ representations of $\GL_2(\F_q)$ and $\omega$ is a mod $\ell$ representation of $\GL_1(\F_q)$. This modified factor is constructed by restricting to subgroups of matrices with \textit{$\ell$-regular} determinant. Namely, the linear functionals giving rise to the gamma factor are defined as sums over these subgroups. In the mirabolic subgroup, the elements with $\ell$-regular determinant have $\ell$-regular order and form a subgroup. The failure of this property for $n > 2$ prevents us from extending the strategy.

For simplicity, unlike in the main part of this article we only construct the $\ell$-regular gamma factors for irreducible cuspidal representations. One could probably also treat Whittaker type representations, taking into account exceptional pairs, but we don't pursue this.

\subsection{Preliminaries}

As before, let $\ell$ be a prime different from $p$ and let $k$ be a field of characteristic $\ell$ that is sufficiently large (this means $k$ contains all the $m$-th roots of unity where $m$ is the l.c.m. of all the orders of elements of $\GL_2(\F_q)$). We write $G_n = \GL_n(\F_q)$ as before, and will focus on $G_2$. We regard $G_1 \subset G_2$ as sitting in the top left. Again we work with the mirabolic subgroup $P_2 = N_2 \rtimes G_1 \subset G_2$. We let $\Pbar_2$ denote the opposite mirabolic subgroup. We fix a nontrivial group homomorphism $\psi: \F_q \to k^\times$, and view it as a character $\psi: N_2 \to k^\times$ via the canonical isomorphism $N_2 \xrightarrow\sim \F_q$.

We now define some auxiliary subgroups. First let $\F_q^{\times,\ell}$ denote the subgroup of~$\F_q^{\times}$ consisting of $\ell$-regular elements, i.e. elements whose orders are not divisible by~$\ell$. Then let
    \[ G_2^\ell := {\det}^{-1}(\F_q^{\times,\ell}) \]
denote the subgroup of matrices with $\ell$-regular determinant, and let $G_1^\ell = G_1 \cap G_2^\ell$ and $P_2^\ell = P_2 \cap G_2^\ell$ and $\Pbar_2^\ell = \Pbar_2 \cap G^\ell_2$. Note $P_2^\ell = N_2 \rtimes G_1^\ell$ and that $G_2=P_2G_2^\ell$. 

\begin{lemma}
The group generated by $P_2^\ell$ and $\Pbar_2^\ell$ is $G_2^\ell$.
\end{lemma}
\begin{proof}
Let $H$ be the subgroup of $G_2$ generated by $P_2^\ell$ and $\Pbar_2^\ell$. Clearly $P_2^\ell \subset G_2^\ell$ and $\Pbar_2^\ell\subset G_2^\ell$, and thus $H\subset G_2^\ell$.

For the opposite inclusion we argue as follows. By row reduction, $\SL_2(\F_q)$ is generated by the elementary matrices with $1$'s on the diagonal and a single nonzero entry off the diagonal. Namely, $\SL_2(\F_q)$ is generated by $N_2$ and $\Ubar_2$. But $N_2 \subset P_2^\ell$ and $\Ubar_2 \subset \Pbar_2^\ell$ so $\SL_2(\F_q) \subset H$. We are done if for every element $a$ of $\F_q^{\times,\ell}$ we can find an element $h \in H$ such that $\det(h) = a$ (for then $H$ contains a full set of representatives for $G_2^\ell/\SL_2(\F_q)$). But we can just take $\diag(a,1)$.
\end{proof}

\begin{theorem}[Clifford's Theorem {{\cite[\S 11.1]{curtis-reiner}}}]\label{Clifford}
If $G$ is a finite group, $H \triangleleft G$, and $\rho$ is an irreducible representation of $G$ over any field $k$, then
    \[
        \restr{\rho}{H} = \bigoplus_{i=1}^r \rho_i^e
    \]
where $\{\rho_i \mid 1 \leq i \leq r\}$ is a set of pairwise non-isomorphic irreducible representations 
of $H$ over $k$.
The $\rho_i$-isotypic components $\rho_i^e$ are permuted transitively under conjugation by $G$
and $\rho = \Ind_{\Stab_G(\rho_i)}^G(\rho_i^e)$ for all $i \in \set{1, \ldots, r}$.
\end{theorem}

Since $G$ acts by conjugation on the set $\set{\rho_1,\dots,\rho_r}$ it follows that $H \subset \stab_G(\rho_i)$. But~$G$ acts transitively so $r = [G : \stab_G(\rho_i)]$, which divides $[G:H]$.

\begin{proposition}\label{WhittakerUnique}
If $(\rho,V)$ is an irreducible generic representation of $G_2$ then
    \[ \dim_k \Hom_{N_2}(V,\psi) = 1. \]
\end{proposition}
\begin{proof}
\cite[\S III.1.7, III.5.10]{vig_book} proves that $\dim_k \Hom_{P_2}(\Ind_{N_2}^{P_2} \psi, V) = 1$. Equivalently, $\dim_k \Hom_{N_2}(\psi, V) = 1$ but $N_2$ is abelian of order prime to $\ell$ so $V|_{N_2}$ splits as the direct sum of characters, so $V|_{N_2}$ contains $\psi$ once. Thus $\dim_k \Hom_{N_2}(V, \psi) = 1$.
\end{proof}

\begin{theorem}[{{\cite[Theorem III.1.1]{vig_book}}}]\label{CuspResInd}
If $(\rho,V)$ is an irreducible cuspidal representation then $V|_{P_2} \cong \Ind_{N_2}^{P_2} \psi$. Furthermore $\Ind_{N_2}^{P_2} \psi$ is irreducible.
\end{theorem}

\subsection{Definition of the \texorpdfstring{$\gamma$}{gamma}-factor} \label{sec:converse-theorem}

Fix an irreducible cuspidal (hence generic) representation $(\rho,V)$ of $G_2$. By \autoref{Clifford} we get a decomposition
    \begin{equation}\label{eq ind decomposition} \rho|_{G_2^\ell} = \bigoplus_{i=1}^r \rho_i^e \end{equation}
where each $(\rho_i,V_i)$ is an irreducible representation of $G_2^\ell$ and $G_2$ permutes them transitively. Moreover,
    \[ \rho = \Ind_{\Stab_{G_2}(\rho_i)}^{G_2} \rho_i^e \]
for any $i$.

\begin{lemma}\label{decomposition}
The restriction $\rho|_{G_2^\ell}$ is multiplicity free. In other words, $e = 1$.
\end{lemma}

\begin{proof}
If $\widehat{N_2}$ denotes the group of characters $N_2 \to k^\times$ then by \autoref{CuspResInd} we have \begin{equation} \label{eq ind decomp N2}\rho|_{N_2} = \bigoplus_{\chi \neq 1 \in \widehat{N_2}} \chi.\end{equation} Each element of $N_2$ is $\ell$-regular, so $N_2 \subset G_2^\ell$ and \eqref{eq ind decomp N2} is a further decomposition of \eqref{eq ind decomposition}. It follows that in the decomposition $\rho|_{G_2^\ell}=A_1\oplus ...\oplus A_s$ into irreducibles, we have that~$A_i \not\cong A_j$ if $i \neq j$ since their restrictions to $N_2$ are not isomorphic.
\end{proof}

By \autoref{WhittakerUnique},
    \[ 1 = \dim_k \Hom_{N_2}(V, \psi) = \dim_k \Hom_{G_2^\ell}(V, \Ind_{N_2}^{G_2^\ell} \psi) = \dim_k \left(\bigoplus_{i=1}^r \Hom_{G_2^\ell} (V_i, \Ind_{N_2}^{G_2^\ell} \psi)\right) \]
so there exists a unique $i_\psi$ such that $\Hom_{G_2^\ell}(V_{i_\psi}, \Ind_{N_2}^{G_2^\ell} \psi) \neq 0$ (and is one dimensional). Write $(\rho_\psi,V_\psi)$ for the representation $(\rho_{i_\psi},V_{i_\psi})$. Fix a generator $W_\psi: V_\psi \hookrightarrow \Ind_{N_2}^{G_2^\ell} \psi$. The image of $W_\psi$ is denoted $\calw^\ell(V_\psi)$ and is called the \textit{$\ell$-regular Whittaker model} of $\rho$.

\begin{proposition}
$\Ind_{N_2}^{P_2^\ell} \psi$ is irreducible.
\end{proposition}
\begin{proof}
Note $N_2$ is a normal subgroup of $P_2^\ell$ with quotient isomorphic to the abelian group $G_1^\ell$, so we can write
    \[ \End_{P_2^\ell}(\Ind_{N_2}^{P_2^\ell} \psi) = \Hom_{N_2}(\psi, \Res_{N_2}^{P_2^\ell} \Ind_{N_2}^{P_2^\ell} \psi) = \Hom_{N_2}\left(\psi, \bigoplus_{g \in G_1^\ell} (x \mapsto \psi(gx))\right), \]
which is clearly one dimensional since $x \mapsto \psi(gx)$ is not equal to $\psi$ for any $g \in G_1^\ell$ except when $g = 1$. Since $P_2^\ell$ has order prime to $\ell$, the result follows from basic character theory.
\end{proof}

\begin{corollary}\label{EllRegularPieceIsInduced}
The composition
    \[ V_\psi \xrightarrow{W_\psi} \Ind_{N_2}^{G_2^\ell} \psi \xrightarrow{\Res^{G_2^\ell}_{P_2^\ell}} \Ind_{N_2}^{P_2^\ell} \psi \]
is an isomorphism of $P_2^\ell$-representations.
\end{corollary}
\begin{proof}
By construction it is a morphism of $P_2^\ell$-representations. Both $V_\psi$ and $\Ind_{N_2}^{P_2^\ell} \psi$ are irreducible, so we just need to show that the composition is nonzero. But the Frobenius reciprocity isomorphism
    \[ \Hom_{G_2^\ell}(V_\psi, \Ind_{N_2}^{G_2^\ell} \psi) \xrightarrow{\sim} \Hom_{P_2^\ell}(V_\psi, \Ind_{N_2}^{P_2^\ell} \psi) \]
is precisely composition with $\Res^{G_2^\ell}_{P_2^\ell}$ and the fact that $W_\psi$ is nonzero means that its image under the above isomorphism is as well.
\end{proof}

\begin{corollary}
For an irreducible cuspidal representation $(\rho,V)$, the number $r$ of irreducible summands in \eqref{eq ind decomposition} satisfies $r = [P_2:P_2^\ell]$. Consequently, $\Stab_{G_2}(\rho_\psi) = G_2^\ell$.
\end{corollary}
\begin{proof}
    \[ r = \frac{\dim_k \rho}{\dim_k \rho_\psi} = \frac{\dim_k \Ind_{N_2}^{P_2} \psi}{\dim_k \Ind_{N_2}^{P_2^\ell} \psi} = \frac{[P_2:N_2]}{[P_2^\ell:N_2]} = [P_2:P_2^\ell]. \]
Since $\rho = \Ind^{G_2}_{\Stab_{G_2}(\rho_\psi)} \rho_\psi$, and since $G_2=P_2G_2^\ell$, we obtain
    \[ [G_2:G_2^\ell] = [P_2:P_2^\ell] = [G_2 : \Stab_{G_2}(\rho_\psi)] \]
so the inclusion $\Stab_{G_2}(\rho_\psi) \subset G_2^\ell$ is an equality.
\end{proof}

Next we prove the key one-dimensionality result that lets us deduce the existence of the gamma factor as the ratio between two linear functionals in a functional equation. Note that because $k$ has characteristic $\ell$ any character $\omega: G_1 \to k^\times$ is uniquely determined by its values on $G_1^\ell$.

\begin{corollary}\label{KeyUniqueness}
For any character $\omega: G_1 \to k^\times$,
    \[ \dim_k \bil_{G_1^\ell}(\calw^\ell(\rho_\psi) \otimes \omega, \one) = 1. \]
\end{corollary}
\begin{proof}
Note 
    \[ \bil_{G_1^\ell}(\calw^\ell(\rho_\psi) \otimes \omega, \one) = \Hom_{G_1^\ell}(V_\psi,\omega^{-1}). \]
By \autoref{EllRegularPieceIsInduced} we have $V_\psi|_{P_2^\ell} \cong \Ind_{N_2}^{P_2^\ell} \psi$. The map
\begin{align*}
    \Ind_{N_2}^{P_2^\ell} \psi &\xrightarrow\sim k[G_1^\ell] \\
    f &\mapsto \left(x \mapsto f\begin{pmatrix}x&0\\0&1\end{pmatrix}\right)
\end{align*}
gives an isomorphism with the regular representation. But $G_1^\ell$ is a cyclic group of order prime to $\ell$ so $k[G_1^\ell]$ contains every $k$-valued character of $G_1^\ell$ with multiplicity one.
\end{proof}

\begin{definition}
Following \autoref{sec:ngm}, for $W \in \calw^\ell(\rho_\psi)$ and $\omega: G_1 \to k^\times$ a character we define
\begin{align*}
    I^\ell(W, \omega) &= \sum_{x \in G_1^\ell} W\begin{psmallmatrix}x&0\\0&1\end{psmallmatrix}\omega(x) \\
    \wt{I}^\ell(W,\omega) &= \sum_{x \in G_1^\ell} W\begin{psmallmatrix}0&1\\x&0\end{psmallmatrix}\omega(x)
\end{align*}
\end{definition}

Then $I^\ell(W, \omega), \wt{I}^\ell(W,\omega) \in \bil_{G_1^\ell}(\calw^\ell(\rho_\psi) \otimes \omega, \one)$ are two nonzero elements. But in view of \autoref{KeyUniqueness} this space is 1-dimensional, so we make the following definition.

\begin{definition}
For $\omega: \F_q^\times \to k^\times$ a character, the \textit{$\ell$-regular gamma factor} $\gamma^\ell(\rho \times \omega,\psi) \in k$ is the unique (nonzero) element satisfying
    \[ I^\ell(W,\omega)\gamma^\ell(\rho \times \omega, \psi) = \wt{I}^\ell(W,\omega). \]
\end{definition}

\subsection{Converse theorem}

We now show that the $\ell$-regular factor satisfies a converse theorem; our strategy mirrors that of \autoref{sec:proof_converse}. Suppose $(\rho_1,V_1)$ and $(\rho_2,V_2)$ are two irreducible cuspidal $k$-linear representations of $G_2$ and further suppose that
    \[ \gamma^\ell(\rho_1 \times \omega,\psi) = \gamma^\ell(\rho_2 \times \omega,\psi) \]
for all $\omega: G_1^\times \to k^\times$. Let $\calw_1^\ell = \calw^\ell(V_{1,\psi})$ and $\calw_2^\ell = \calw^\ell(V_{2,\psi})$.

\begin{definition}
Let
    \[ S(\rho_1,\rho_2,\psi) := \set{(W_1, W_2) \in \calw^\ell_1 \times \calw^\ell_2 : W_1|_{P_2^\ell} = W_2|_{P_2^\ell}}. \]
\end{definition}

By definition there is a diagonal action of $G_2^\ell$ on $S(\rho_1,\rho_2,\psi)$ and $S(\rho_1,\rho_2,\psi)$ is $P_2^\ell$-stable for this action.

\begin{lemma}\label{lem:S-stability}
If $g \in G_2^\ell$ and $(W_1, W_2) \in S(\rho_1,\rho_2,\psi)$, then
$(gW_1, gW_2) \in S(\rho_1,\rho_2,\psi)$.
\end{lemma}
\begin{proof}
    First note that since $W_1,W_2$ are Whittaker functions,
    \begin{align*}
        (W_1,W_2) \in S(\rho_1,\rho_2,\psi) &\iff W_1\begin{psmallmatrix}x&0\\0&1\end{psmallmatrix} = W_2\begin{psmallmatrix}x&0\\0&1\end{psmallmatrix} \text{ for all } x \in G_1^\ell \\
        &\xLeftrightarrow{\text{Artin's Lemma}} I^\ell(W_1,\omega) = I^\ell(W_2,\omega) \text{ for all } \omega \\
        &\xLeftrightarrow{\text{equality of $\gamma^\ell$}} \wt{I}^\ell(W_1,\omega) = \wt{I}^\ell(W_2,\omega) \text{ for all } \omega \\
        &\xLeftrightarrow{\text{Artin's Lemma}} W_1\begin{psmallmatrix}0&1\\x&0\end{psmallmatrix} = W_2\begin{psmallmatrix}0&1\\x&0\end{psmallmatrix} \text{ for all } x \in G_1^\ell \\
        &\iff \wt{W_1}\begin{psmallmatrix}x&0\\0&1\end{psmallmatrix} = \wt{W_2}\begin{psmallmatrix}x&0\\0&1\end{psmallmatrix} \text{ for all } x \in G_1^\ell \\
        &\iff (\wt{W_1},\wt{W_2}) \in S(\rho_1^\vee,\rho_2^\vee,\psi^{-1})
    \end{align*}
    Here Artin's Lemma (the $n=1$ version of completeness of Whittaker models) refers to the dual statement to linear independence of characters, which holds for $k$-valued characters of an abelian group $H$, provided that $\mathrm{char}(k) \nmid \verts{H}$, see \cite[\S 1]{ps}. Now if $\overline{p} \in \Pbar_2^\ell$ and $W \in \calw(\rho_i,\psi)$ (for $i = 1,2$), then for all $g \in G_2^\ell$ we have, as in \eqref{whit model transpose}, 
    \[  
    \wt{\overline{p}W}(g) = (^\iota\overline{p}\wt{W})(g).
    \]
    Thus if $(W,W') \in S(\rho_1,\rho_2,\psi)$ then
        \[ (\wt{\overline{p}W}, \wt{\overline{p}W'}) = (^\iota\overline{p}\wt{W}, ^\iota\overline{p}\wt{W'}) \in S(\rho^\vee,\sigma^\vee,\psi^{-1}) \]
    since $S(\rho^\vee,\sigma^\vee,\psi^{-1})$ is $P_2^\ell$-stable and $^\iota\overline{p} \in P_2^\ell$. By the above equivalences we see that $(\overline{p}W,\overline{p}W') \in S(\rho_1,\rho_2,\psi)$. We conclude by noting that $P_2^\ell$ and $\Pbar_2^\ell$ generate $G_2^\ell$.
\end{proof}

Following the same ideas as the proof of
\autoref{corollary restriction to Pn}, we deduce:

\begin{corollary}
    \label{restriction_to_pell}
If $W_1 \in \calw^\ell_1$ and $W_2 \in \calw^\ell_2$, then
    \[
        \restr{W_1}{P_2^\ell} = \restr{W_2}{P_2^\ell} \textrm{ if and only if } W_1 = W_2.
    \]
\end{corollary}

\begin{theorem}
    If $\gamma^\ell(\rho_1 \times \omega,\psi) = \gamma^\ell(\rho_2 \times \omega,\psi)$
    for all $\omega: G_1^\times \to k^\times$, then
    $\rho_1 \cong \rho_2$.
\end{theorem}
\begin{proof}
    Since $\rho_1 = \Ind_{G_2^\ell}^{G_2} \rho_{1,\psi}$ and $\rho_2 = \Ind_{G_2^\ell}^{G_2} \rho_{2,\psi}$, it suffices to show that $\calw_1^\ell = \calw_2^\ell$, since then $\rho_{1,\psi} \cong \rho_{2,\psi}$. But $\calw_1^\ell|_{P_2} = \calw_2^\ell|_{P_2}$, so we apply \autoref{restriction_to_pell} to conclude.
\end{proof}

\end{appendices}


\printbibliography


\end{document}